\documentclass{amsart}
\usepackage{amsmath}
\usepackage{amssymb}
\usepackage[left=1.5in,right=1.5in,bottom=1.5in,top=1.5in]{geometry}

\newcommand{\map}[1]{\;\xrightarrow{#1}\;}
\newcommand{\mil}{\varprojlim}
\newcommand{\dlim}{\varinjlim}
\newcommand{\iso}{\cong}
\newcommand{\stack}[2]{\;\genfrac{}{}{0pt}{}{#1}{#2}\;}

\newcommand{\Gal}{\mathrm{Gal}}
\newcommand{\Hom}{\mathrm{Hom}}
\newcommand{\End}{\mathrm{End}}
\newcommand{\alg}{\mathrm{alg}}
\newcommand{\Norm}{\mathrm{Norm}}
\newcommand{\N}{\mathbf{N}}

\newcommand{\Pic}{\mathrm {Pic}}
\newcommand{\Spec}{\mathrm {Spec}}
\newcommand{\ord}{\mathrm {ord}}
\newcommand{\Q}{\mathbf Q}
\newcommand{\Z}{\mathbf Z}

\newcommand{\C}{\mathbf C}
\newcommand{\co}{\mathcal O}
\newcommand{\cA}{\mathcal A}
\newcommand{\cB}{\mathcal B}
\newcommand{\cL}{\mathcal L}
\newcommand{\fa}{\mathfrak a}
\newcommand{\fb}{\mathfrak b}
\newcommand{\diff}{\mathfrak D}
\newcommand{\Hecke}{\mathbf T}

\newcommand{\stable}{\mathrm {stable}}
\newcommand{\unstable}{\mathrm {unstable}}
\newcommand{\fn}{\mathfrak n}
\newcommand{\fq}{\mathfrak q}
\newcommand{\fg}{\mathfrak g}
\newcommand{\fl}{\mathfrak l}
\newcommand{\uh}{ {\underline h} }
\newcommand{\uE}{ {\underline E} }
\newcommand{\uX}{ {\underline X}}
\newcommand{\uC}{{\underline C}}
\newcommand{\uD}{{\underline D}}

\newcommand{\bc}{\mathbf c}
\newcommand{\bd}{\mathbf d}
\newcommand{\bh}{\mathbf h}
\newcommand{\ubh}{ {\underline \bh} }
\newcommand{\uz}{{\underline z}}
\newcommand{\uy}{{\underline y}}

\newcommand{\anti}{\mathrm{anti}}

\input xy
\xyoption{all}

\begin{document}
\title{The Iwasawa theoretic Gross-Zagier theorem}
\author{Benjamin Howard}
\thanks{This research was supported by an NSF postdoctoral fellowship.}
\address{Department of Mathematics, Harvard University, Cambridge, MA.}
\curraddr{Department of Mathematics, University of Chicago, Chicago, IL.}

\begin{abstract}
We prove Mazur and Rubin's $\Lambda$-adic Gross-Zagier conjecture (under 
some restrictive hypotheses), which relates Heegner points in towers of
number fields to the $2$-variable $p$-adic $L$-function.
 The result generalizes Perrin-Riou's $p$-adic
Gross-Zagier theorem.  
\end{abstract}

\maketitle

\theoremstyle{plain}
\newtheorem{Thm}{Theorem}[subsection]
\newtheorem{Prop}[Thm]{Proposition}
\newtheorem{Lem}[Thm]{Lemma}
\newtheorem{Cor}[Thm]{Corollary}
\newtheorem{Con}[Thm]{Conjecture}
\newtheorem{BigTheorem}{Theorem}

\theoremstyle{definition}
\newtheorem{Def}[Thm]{Definition}

\theoremstyle{remark}
\newtheorem{Rem}[Thm]{Remark}

\renewcommand{\labelenumi}{(\alph{enumi})}
\renewcommand{\theBigTheorem}{\Alph{BigTheorem}}
\setcounter{section}{-1}


\section{Introduction}


Fix forever a rational prime $p>2$ and  embeddings
$\Q^\alg\hookrightarrow\Q_p^\alg$ and $\Q^\alg\hookrightarrow\C$.  
Fix also a normalized cuspidal newform
$f\in S_2(\Gamma_0(N),\C)$ and an imaginary quadratic field $K/\Q$
of discriminant $D$ and quadratic character $\epsilon$
satisfying the Heegner hypothesis 
that all primes dividing $N$ are split in $K$.  Assume that $(p,DN)=1$
and that $f$ is \emph{ordinary} at $p$ in the sense that the 
Fourier coefficient $a_p(f)\in\Q^\alg$ has $p$-adic absolute value $1$
at the fixed embedding $\Q^\alg\hookrightarrow\Q_p^\alg$.
We let $\cB_0$ be a number field
which is large enough to contain all Fourier coefficients of $f$,
denote by $\cA_0$ the integer ring of $\cB_0$, and by $\cA$ and $\cB$
the closures of $\cA_0$ and $\cB_0$ in $\Q_p^\alg$, respectively.
Let $H_s$ be the ring class field of $K$ of conductor $p^s$ and
let $H_\infty$ be the union over all $s$ of $H_s$.
We write $\Gamma=1+p\Z_p$,
and let $\gamma_0\in\Gamma$ be a topological generator.
Using methods of Hida \cite{Hida}, Perrin-Riou \cite{pr1, pr0}
attaches to $f$ a 
``two-variable'' $p$-adic $L$-function
$$\cL_f\in\cA[[\Gal(H_\infty/K)\times\Gamma]]\otimes_\cA\cB$$
which interpolates the special values of twists of the complex $L$-function of
$f$ at $s=1$.  The $p$-adic $L$-function
may be expanded as a power series in $\gamma_0-1$
\begin{equation}\label{two variable decomposition}
\cL_f=\cL_{f,0}+\cL_{f,1}\cdot(\gamma_0-1)+\ldots,
\end{equation}
with each $\cL_{f,k}\in \cA[[\Gal(H_\infty/K)]]\otimes_\cA\cB$.
The Heegner hypothesis forces the constant term $\cL_{f,0}$
to vanish, and  the goal of this paper is to relate the linear term
$\cL_{f,1}$ to the $p$-adic height pairings of Heegner points
in the $f$-component of the Jacobian $J_0(N)$.

For every nonnegative integer $s$ the Heegner hypothesis guarantees the
existence of a Heegner point $h_s\in X_0(N)(\C)$
of conductor $p^s$; that is, a cyclic $N$-isogeny of elliptic
curves $h_s:E_s\map{}E'_s$ over $\C$ such that both $E_s$ and $E'_s$ 
have complex multiplication
by \emph{exactly} $\co_s=\Z+p^s\co_K$, the order of conductor $p^s$ in $K$.
The family $\{h_s\}$ may be chosen
so that for every $s>1$ there is a  commutative diagram
$$
\xymatrix{
E_s\ar[r]^{h_s}\ar[d]&E_s'\ar[d]\\
E_{s-1}\ar[r]^{h_{s-1}}&E_{s-1}'
}$$
in which the vertical arrows are $p$-isogenies.
The elliptic curve $E_{s-1}$ (resp. $E'_{s-1}$) is then necessarily
the quotient of $E_s$ (resp. $E_s'$) by its $p\co_{s-1}$-torsion.
By the theory of complex multiplication (for example
\cite[Proposition 1.2]{cornut}) the curves $E_s$ and $E'_s$,
as well as the isogeny connecting them, can be defined over $H_s$, 
and so define a point $h_s\in X_0(N)(H_s)$.
One then has the Euler system relations (\S \ref{Hecke action})
$$
T_{p^{r}}(h_s)=\Norm_{H_{s+r}/H_s}(h_{s+r})+T_{p^{r-1}}(h_{s-1})$$
if $r,s>0$, and 
$$
T_p(h_0)=
\left\{\begin{array}{ll}
u\cdot \Norm_{H_1/H_0}(h_1)+ (\sigma_p+\sigma_p^*)h_0
& \mathrm{if\ } \epsilon(p)=1\\
u\cdot \Norm_{H_1/H_0}(h_1)& \mathrm{if\ } \epsilon(p)=-1
\end{array}\right.
$$
as divisors on $X_0(N)$, where $T_{p^r}$ is the usual Hecke
correspondence, $2u=|\co_K^\times|$, and $\sigma_p,\ \sigma_p^*\in\Gal(H_0/K)$
are the Frobenius automorphisms of the two primes above $p$ in the case
$\epsilon(p)=1$.
Abusing notation, we also denote by $h_s$ the image of $h_s$ in
$J_0(N)$ under the usual embedding $X_0(N)\map{}J_0(N)$
taking the cusp $\infty$ to the origin.

Let $\Hecke$ be the $\Q$-algebra generated by the action 
of the Hecke operators $T_\ell$ with $(\ell,N)=1$ on 
$J_0(N)$.  The semi-simplicity of $\Hecke$ gives a
decomposition of $\Hecke\otimes\cB$-modules
$$J_0(N)(H_s)\otimes_\Z\cB\iso\bigoplus_\beta J(H_s)_\beta$$
where $\beta$ ranges over $\Gal(\Q_p^\alg/\cB)$-orbits of
algebra homomorphisms $\beta:\Hecke\map{}\Q_p^\alg$.
Each summand is stable under the action of $\Gal(H_s/\Q)$,
and if $\beta(\Hecke)\subset\cB$ then 
$\Hecke$ acts on $J(H_s)_{\beta}$ through the character $\beta$.
The fixed newform $f$ determines one such homomorphism,
and we define $h_{s,f}$ to be the projection of
$h_s$ onto the associated factor $J(H_s)_f$.  Let $\alpha\in\cA^\times$ 
be the unit root of $X^2-a_p(f)X+p$. As in \cite{bd mumford}, define 
the \emph{regularized Heegner point}
$z_s\in J(H_s)_f$ for $s>0$ by
$$z_s=\frac{1}{\alpha^s}h_{s,f}-\frac{1}{\alpha^{s+1}}h_{s-1,f}.$$
In the case $s=0$ we define
$$
z_0=u^{-1}\cdot\left\{\begin{array}{ll}
\big(1-\frac{\sigma_p}{\alpha}\big)
\big(1-\frac{\sigma_p^*}{\alpha}\big) h_{0,f} & \mathrm{if\ }
\epsilon(p)=1\\ \\
\big(1-\frac{1}{\alpha^{2}}\big)h_{0,f}
&\mathrm{if\ } \epsilon(p)=-1.\end{array}
\right.$$
It follows from the Euler system relations
that the points $z_s$ are compatible under the norm (trace) maps
on $J(H_s)_f$.

The case $s=0$ of the following theorem is due to Perrin-Riou \cite{pr1},
and has been generalized to higher weight modular forms by 
Nekov\'{a}\v{r} \cite{nekovar}.

\begin{BigTheorem}\label{main result}
Assume that $D$ is odd and $\not=-3$, and that $\epsilon(p)=1$. 
For any character $\eta:\Gal(H_s/K)\map{}\Q_p^{\alg,\times}$
$$
\eta(\kappa_s)\log_p(\gamma_0)\cdot \cL_{f,1}(\eta)=
\sum_{\sigma\in\Gal(H_s/K)}\eta(\sigma)
\langle z_s^\vee,z_s^\sigma\rangle
$$
where $\kappa_s\in\Gal(H_s/K)$ is the Artin symbol of 
$\mathfrak{d}_s=(\sqrt{D}\co_K)\cap \co_s$,
$$
\langle\ ,\ \rangle=\langle\ ,\ \rangle_{J_0(N),H_s}
:J_0(N)^\vee(H_s)\times J_0(N)(H_s)\map{}\Q_p
$$ 
is the $p$-adic height pairing (\ref{global abelian}) extended
$\cB$ bilinearly, and $z_s^\vee$ is the image of $z_s$
under the canonical principal polarization of $J_0(N)$ (extended
$\cB$-linearly on Mordell-Weil groups)
$$
J_0(N)(H_s)\otimes\cB\iso J_0(N)^\vee(H_s)\otimes\cB.
$$
Both sides of the stated equality are independent of the choice of $\gamma_0$.
\end{BigTheorem}

\begin{Rem}
The $p$-adic height pairing $\langle\ ,\ \rangle_{J_0(N),H_s}$ 
referred to in the theorem is not uniquely determined
(see Proposition \ref{pr height} and Remark \ref{base compatible}).
We emphasize that Theorem \ref{main result} holds
for \emph{any} choice of $p$-adic height pairing
$\langle\ ,\ \rangle_{J_0(N),H_s}$ as in (\ref{global abelian}).
\end{Rem}

\begin{Rem}
Nekov\'{a}\v{r} \cite{nekovar} claims that there is a sign error
in the statement of \cite[Th\'eor\`eme 1.3]{pr1}, but there
is no small amount of confusion over Perrin-Riou's normalization 
of the height pairing.  This is primarily due to 
the change of sign in Remark \ref{Theta}, which is our reason 
for maintaining the distinction between $J_0(N)$ and $J_0(N)^\vee$,
and between the pairings (\ref{global abelian}) and (\ref{global curve}).
It is also possible that \cite{pr1} uses a different convention for 
the reciprocity law of class field theory; see \S \ref{global pairing}.
\end{Rem}

\begin{Rem}
Theorem \ref{main result} should hold without the stated 
hypotheses on $D$ and $\epsilon(p)$.
We note that the hypothesis $D\not=-3$ 
is not assumed in \cite{pr1}.
\end{Rem}

Now suppose $f$ has rational Fourier coefficients, $\cB_0=\Q$, and $E$
belongs to the isogeny class of (ordinary!) elliptic curves
associated to $f$. Fix a modular parametrization $X_0(N)\map{\phi}E$,
and let 
$$
\phi_*:J_0(N)\map{}E
\hspace{1cm}
\phi^*:E^\vee\map{}J_0(N)^\vee
$$
be the Albanese and Picard maps.
Extending $\phi_*$ and $\phi^*$ to $\Q_p$-linear maps on Mordell-Weil groups, 
let $y_s=\phi_*(z_s)\in E(H_s)\otimes\Z_p$ and let $y_s^\vee$ be the 
unique point of $E^\vee(H_s)\otimes\Q_p$ with
$\phi^*(y_s^\vee)=z_s^\vee$.
The canonical polarization $E\iso E^\vee$ identifies
$y_s$ with $\deg(\phi)\cdot y_s^\vee$.  The points $y_s$ and $y_s^\vee$
are norm-compatible as $s$ varies (since the $z_s$ are).
Define the Heegner $L$-function
$\mathcal{L}_\mathrm{Heeg}\in \Z_p[[\Gal(H_\infty/K)]]\otimes \Q_p$
by 
$$
\mathcal{L}_\mathrm{Heeg}= \mil \sum_{\sigma\in\Gal(H_s/K)}
\langle y_s^\vee,y_s^\sigma\rangle_{E,H_s}\cdot \sigma
$$
where the pairing is the $p$-adic height pairing of 
(\ref{global abelian}) extended $\Q_p$-linearly 
(and \emph{not} the height pairing of
(\ref{global curve}); as $E$ is both a curve and an abelian variety,
we have reached a notational singularity). 
Unlike the height pairing of Theorem \ref{main result}, the
pairing $\langle\ ,\ \rangle_{E,H_s}$ is canonical.  This
follows from the  ordinarity of  $E$ at $p$ and the 
uniqueness claims of Proposition \ref{pr height}.
A priori, $\mathcal{L}_\mathrm{Heeg}$ lives in the larger space
$\mil \Q_p[[\Gal(H_s/K)]]$, but it is known that the denominators
in the height pairing are bounded as $s$ varies (this follows from 
the construction of \cite{pr1}, although it is not explicitly stated
there; note also Proposition \ref{bounded denominators} below).

\begin{BigTheorem}\label{main result II}
Under the hypotheses (and notation) of Theorem \ref{main result},
$$
\kappa\cdot \log_p(\gamma_0)\cdot \cL_{f,1} = \mathcal{L}_\mathrm{Heeg}
$$
in $\Z_p[[\Gal(H_\infty/K)]]\otimes \Q_p$, where 
$\kappa=\mil \kappa_s\in\Gal(H_\infty/K)$.
\end{BigTheorem}

Theorem \ref{main result II} is a (very slightly) strengthened form
of a conjecture of Mazur and Rubin \cite[Conjecture 9]{MR}.
To make the connection between our theorem and the conjecture of 
Mazur and Rubin more explicit, first note that the construction
of the $p$-adic height $\langle\ ,\ \rangle_{E,H_s}$ depends on the
auxillary choice of the idele class character 
$\rho_{H_s}:\mathbf{A}_{H_s}^\times/H_s^\times
\map{}\Gamma\map{\log_p}\Z_p$ defined at the start of \S \ref{global pairing}.
Define $\Gamma_{\Q_p}=\Gamma\otimes_{\Z_p} \Q_p$ and
extend $\log_p$ to a $\Q_p$-linear isomorphism 
$\Gamma_{\Q_p}\iso\Q_p$.
Define a pairing
$$
\langle\ ,\ \rangle^\Gamma_{E,H_s}: E^\vee(H_s)\times E(H_s)\map{}\Gamma_{\Q_p}
$$
by $\langle\ ,\ \rangle_{E,H_s}=\log_p\circ \langle\ ,\ \rangle^\Gamma_{E,H_s}$
and set
$$
\mathcal{L}^\Gamma_{\mathrm{Heeg}}=\mil \sum_{\sigma\in\Gal(H_s/K)}
\langle y_s,y_s^\sigma\rangle^\Gamma_{E,H_s}\cdot \sigma
\in \Z_p[[\Gal(H_\infty/K)]]\otimes\Gamma_{\Q_p},
$$
where we have now identified $E\iso E^\vee$ in the canonical way,
so that 
$$
(1\otimes\log_p)(\mathcal{L}^\Gamma_\mathrm{Heeg})=
\deg(\phi)\cdot \mathcal{L}_\mathrm{Heeg}.
$$
Let $I$ be the kernel of the projection 
$$
\Z_p[[\Gal(H_\infty/K)\times\Gamma]]\otimes \Q_p
\map{}\Z_p[[\Gal(H_\infty/K)]]\otimes \Q_p
$$
and let $w:\Z_p[[\Gal(H_\infty/K)]]\otimes \Gamma_{\Q_p} \map{} I/I^2$ 
be the isomorphism defined by  
$w(\lambda\otimes \gamma)=\lambda (\gamma-1)$.
Thus 
$w(\cL_\mathrm{Heeg}^\Gamma)=\deg(\phi)
\log_p(\gamma_0)^{-1}\cL_\mathrm{Heeg}\cdot (\gamma_0-1)$.
As $\cL_{f,0}=0$, the $p$-adic $L$-function $\cL_f$ is contained in 
$I$, and Theorem \ref{main result II} may be rewritten as
$$
\kappa\cdot\cL_f=\kappa\cdot\cL_{f,1}\cdot(\gamma_0-1)=
\frac{1}{\log_p(\gamma_0)}\cL_\mathrm{Heeg}\cdot(\gamma_0-1)= 
\frac{1}{\deg(\phi)}w(\mathcal{L}^\Gamma_{\mathrm{Heeg}})
$$
in $I/I^2$.

Now assume the hypotheses of Theorem \ref{main result}, and also 
that $\Gal(K^\alg/K)$ surjects onto the $\Z_p$-module automorphisms of
$T_p(E)$ and that $p$ does not divide the
class number of $K$. Let $K_\infty\subset H_\infty$ be the
anticyclotomic $\Z_p$-extension of $K$, and set $K_s=K_\infty\cap H_{s+1}$,
so that $[K_s:K]=p^{s}$. 
Define $\Lambda_\anti=\Z_p[[\Gal(K_\infty/K)]]\otimes\Q_p$, and
$$ 
\mathcal{S}(K_s,E)=\mil_k\mathrm{Sel}_{p^k}(K_s,E)\hspace{1cm}
\mathcal{S}_\infty=(\mil_s \mathcal{S}(K_s,E))\otimes\Q_p
$$
$$
X=\Hom_{\Z_p}(\mathrm{Sel}_{p^\infty}(K_\infty,E),\Q_p/\Z_p)\otimes\Q_p.
$$
Let $\tilde{y}_\infty\in \mathcal{S}_\infty$ be the inverse limit of 
$\tilde{y}_s=\Norm_{H_{s+1}/K_s}(y_{s+1})\in \mathcal{S}(K_s,E)$, and define
the \emph{Heegner submodule}  $\mathcal{H}\subset \mathcal{S}_\infty$ to be
the $\Lambda_\anti$-submodule generated by $\tilde{y}_\infty$.
It follows from work of Cornut and Vatsal \cite{cornut} 
that $\mathcal{H}$ is a free 
$\Lambda_\anti$-module of rank one.
It is known by work of Bertolini and the author \cite{bert, me}
that $X$ is a finitely-generated rank-one 
$\Lambda_\anti$-module, 
$\mathcal{S}_\infty$ is free of rank one, and
\begin{equation}\label{main conjecture}
\mathrm{char}(X_\mathrm{tors})\mathrm{\ \ divides\ \ }
\mathrm{char}(\mathcal{S}_\infty/\mathcal{H})\cdot 
\mathrm{char}(\mathcal{S}_\infty/\mathcal{H})^\iota
\end{equation}
where $X_\mathrm{tors}$ denotes the $\Lambda_\anti$-torsion submodule
of $X$, and $\lambda\mapsto\lambda^\iota$ is the involution of $\Lambda_\anti$
which is inversion on group-like elements.
Perrin-Riou \cite[Conjecture B]{pr2} has conjectured 
that the divisibility (\ref{main conjecture}) is an equality.

\begin{Prop}\label{bounded denominators}(Perrin-Riou \cite{pr2,pr3,pr4})
There is a $p$-adic height pairing
$$
\mathfrak{h}_s:\mathcal{S}(K_s,E)\times \mathcal{S}(K_s,E)\map{}c^{-1}\Z_p
$$
whose restriction to the image of the Kummer map
$E(K_s)\otimes\Z_p\map{}\mathcal{S}(K_s,E)$ agrees with the pairing
$\langle\ ,\ \rangle_{E,K_s}$ of (\ref{global abelian})
after identifiying $E\iso E^\vee$ in the canonical way,
where $c\in\Z_p$ is independent of $s$.
\end{Prop}

There is a $\Lambda_\anti$-adic height pairing
$\mathfrak{h}_\infty:\mathcal{S}_\infty\times \mathcal{S}_\infty
\map{}\Lambda_\anti$
defined by 
$$
\mathfrak{h}_\infty(\mil a_s,\mil b_s)=\mil \sum_{\sigma\in\Gal(K_s/K)} 
\mathfrak{h}_s(a_s,b_s^\sigma)\cdot \sigma,
$$
and we define the $\Lambda_\anti$-adic
regulator $\mathcal{R}$ to be the image of this map.
If $$e:\Z_p[[\Gal(H_\infty/K)]]\otimes\Q_p\map{}\Lambda_\anti$$ is 
the natural projection, then the norm compatibility of the 
height pairing (see Remark \ref{base compatible}; in this case the 
compatibility is \emph{automatic} by the uniqueness claim of
Proposition \ref{pr height} and the fact that $E$ is ordinary at $p$)
gives 
$$
e(\cL_\mathrm{Heeg})\Lambda_\anti=
\mathfrak{h}_\infty(\tilde{y}_\infty, \tilde{y}_\infty)\Lambda_\anti
=\mathrm{char}(\mathcal{S}_\infty/\mathcal{H})\cdot 
\mathrm{char}(\mathcal{S}_\infty/\mathcal{H})^\iota \cdot \mathcal{R}.
$$ 
If we assume $\mathcal{R}\not=0$ then Theorem \ref{main result II}
allows us to rewrite the divisibility
(\ref{main conjecture}) as
\begin{equation}\label{mc II}
\mathrm{char}(X_\mathrm{tors})\mathrm{\ \ divides\ \ }
\frac{e(\cL_{f,1})\Lambda_\anti}{\mathcal{R}},
\end{equation}
which now has the look and feel of a  $\Lambda_\anti$-adic form
of the Birch and Swinnerton-Dyer conjecture
and no longer makes any mention of Heegner points.  
It was conjectured by Mazur and Rubin \cite[Conjecture 6]{MR} 
that $\mathcal{R}=\Lambda_\anti$, but those authors 
have since retracted that conjecture.

Note that the hypothesis on the action of Galois on the
$p$-adic Tate module excludes the case where $E$ has complex multiplication.
Results similar to (\ref{mc II}) in the so-called exceptional 
case where $E$ has complex multiplication by $K$ can be found in  
\cite{agboola}.

The author thanks Dick Gross for several helpful conversations, 
Brian Conrad for helpful correspondence, and the anonymous 
referee for suggesting many improvements to an earlier draft of
this article.


\subsection{Plan of the proof}
\label{the proof}


Enlarging $\cB_0$ if needed, we may assume that $\cA_0$
contains the Fourier coefficients of all normalized newforms of
level dividing $N$, so that all algebra maps $\Hecke\map{}\Q^\alg$
take values in $\cB_0$.
Fix $s>0$ and define, for each integer $0\le i\le s$, 
degree $0$ divisors on $X_0(N)_{/H_s}$
$$
c_i=(h_i)-(0)\hspace{1cm}d_i=(h_i)-(\infty).
$$
For any pair $0\le i,j\le s$ and any $\sigma\in\Gal(H_s/K)$
we define a $p$-adic modular form
$$
F_\sigma^{i,j}=\sum_\beta \langle c_{i}, 
d_{j,\beta}^\sigma\rangle f_\beta  \ \in   
S_2(\Gamma_0(N),\cB_0)\otimes_{\cB_0}\cB
$$
where the sum is over algebra homomorphisms 
$\beta:\Hecke\map{}\cB_0$, $f_\beta$ is the associated normalized
primitive (i.e. new of some level dividing $N$) eigenform,
$\langle\ ,\ \rangle=\langle\ ,\ \rangle_{X_0(N),H_s}$ 
is the $p$-adic height pairing (\ref{global curve})
on degree zero divisors of $X_0(N)_{/H_s}$ (viewed as a pairing on 
$J_0(N)(H_s)$ and extended $\cB$-linearly; by Remark \ref{Theta}
this is \emph{minus} the pairing of Theorem \ref{main result})
and the $\beta$ subscript on $d_j$
indicates projection to the component $J(H_s)_\beta$.
Define a $p$-adic cusp form
$$
F_\sigma= U^2 F_\sigma^{s,s}-U F_\sigma^{s,s-1}-UF_\sigma^{s-1,s}
+F_\sigma^{s-1,s-1}\ \ \in S_2(\Gamma_0(Np),\cB_0)\otimes_{\cB_0}\cB
$$
where $U$ is the Atkin-Lehner $U_p$ defined by
$U(\sum a_m q^m)=\sum a_{mp} q^m$.
For $(m,N)=1$, the $m^\mathrm{th}$ Fourier coefficient of 
$F_\sigma$ is given by the formula (see Proposition \ref{F expansion}) 
\begin{equation}\label{fourier}
a_m(F_\sigma)=\langle c_s, T_{mp^2}(d_s^\sigma)\rangle
-\langle c_s, T_{mp}(d_{s-1}^\sigma)\rangle  
-\langle c_{s-1}, T_{mp}(d_s^\sigma)\rangle
+\langle c_{s-1}, T_m(d_{s-1}^\sigma)\rangle.
\end{equation}
The pairs of divisors occuring in this expression will not be
relatively prime for many values of $m$, but if we 
define divisors
$$
\bh_{s,r}=\Norm_{H_{s+r}/H_s}(h_{s+r})\hspace{1cm}
\bd_{s,r}=\Norm_{H_{s+r}/H_s}(d_{s+r})$$ 
on $X_0(N)$
and write $m=m_0p^r$ with $(m_0,p)=1$, then the Euler system relation
allow us to rewrite (\ref{fourier}) as 
\begin{equation}\label{fourier 2}
a_m(F_\sigma)=\langle c_s, T_{m_0}(\bd_{s,r+2}^\sigma)\rangle
-\langle c_{s-1}, T_{m_0}(\bd_{s,r+1}^\sigma)\rangle.
\end{equation}
The pairs of divisors occuring here are relatively prime:
the geometric points of $T_{m_0}(\bh_{s,r})$ represent 
elliptic curves with CM by an order $\co$ for which 
$\ord_p(\mathrm{cond}(\co))=r+s.$
Working with these divisors
allows us to avoid the ``intersection theory with tangent vectors'' used
by Gross-Zagier to deal with divisors having common support.

In \S \ref{L function} we recall some $p$-adic analytic
results of Hida and Perrin-Riou. In particular, we recall the construction 
of a $p$-adic modular form $G_\sigma\in M_2(\Gamma_0(Np^\infty),\cA)$ 
(a space defined at the beginning of  \S\ref{L function})
for each $\sigma\in\Gal(H_s/K)$, with the property that 
$$
\log_p(\gamma_0)\cdot \cL_{f,1}(\eta)=
\sum_{\sigma\in\Gal(H_s/K)} \eta(\sigma) L_f(G_\sigma)
$$
for every character $\eta$ of $\Gal(H_s/K)$.  Here $L_f$
is a linear functional $$L_f:M_2(\Gamma_0(Np^\infty),\cA)\map{}\cB$$
which plays the Hida-theoretic role of taking the Petersson 
inner product with $f$.  

Perrin-Riou gives an explicit formula for the Fourier coefficient
$a_m(G_\sigma)$ when  $p$ divides $m$ (Proposition \ref{pr calculation}), 
and in
Sections \ref{modular intersections}, \ref{nonsplit}, and 
\ref{p neron} we adapt the methods of Gross-Zagier 
and Perrin-Riou  to
compute (to the extent necessary) the Fourier coefficients of
$F_\sigma$.  More precisely, each Fourier coefficient has a 
decomposition over the finite places of $H_s$,
$a_m(F_\sigma)=\sum_v a_m(F_\sigma)_v$,
arising from the decomposition of the $p$-adic heights in
(\ref{fourier 2}) into local $p$-adic N\'eron symbols on $X_0(N)_{/H_{s,v}}$.
For $v$ lying above a rational prime $\not=p$ which splits in $K$, 
$a_m(F_\sigma)_v=0$ (Proposition \ref{split primes}).  
For $v$ above a nonsplit rational
prime $\ell\not=p$ we derive an explicit formula (Proposition
\ref{nonsplit evaluation}) for  
$\sum_{v\mid \ell}a_m(F_\sigma)_v$ similar to formulas of Gross-Zagier.
For $v\mid p$ we can offer no explicit formula for $a_m(F_\sigma)_v$,
instead we show that the contribution of $a_m(F_\sigma)_p$ to 
$a_m(F_\sigma)$ is killed by the operator $L_f$ (Proposition
\ref{annihilation}).  This is where we must
impose the condition $\epsilon(p)=1$, although 
Proposition \ref{annihilation} should also hold when $\epsilon(p)=-1$.
Comparing these calculations with the Fourier coefficients of $G_\sigma$,
we conclude that 
$$
L_f( U^{2s}(1-U^2)G_{\sigma\kappa})=L_f(F_\sigma),
$$
and Theorems \ref{main result} and \ref{main result II}
follow easily (see \S \ref{fin} for the 
details).


\subsection{Notation and conventions}
\label{notation}


The data $K$, $p$, $N$, $D$, $f$, $\cA_0$, and $\{h_s\}$ are fixed throughout.
We continue to assume, as in \S \ref{the proof}, that $\cA_0$ contains
the Fourier coefficients of all normalized primitive forms of level $N$.
We typically do \emph{not} assume that $D$ is odd or $\not= -3, -4$,
or that $\epsilon(p)=1$,
unless explicitly stated otherwise.
The parity assumption on $D$ 
is needed only for the results of Perrin-Riou cited in
\S \ref{L function}. The condition $\epsilon(p)=1$ and $D\not=-3,-4$
is used in the calculation
of local N\'eron symbols above $p$ in \S \ref{p neron}.

If $M$ is any $\Z$-module of finite type and $r$ is a rational prime
we set $M_r=M\otimes_\Z\Z_r$.  For any integer $n$, any order $\co\subset K$,
 and any proper fractional $\co$-ideal $\fa$,
we denote by $r_\fa(n)$ the number
of proper, integral $\co$-ideals of norm $n$ whose class
in $\Pic(\co)$ agrees with that of $\fa$. The order $\co$
will usually be clear from the context.  If there is any ambiguity we
will write $r_{\fa \co}(n)$.  Since
complex conjugation acts by inversion on $\Pic(\co)$,
$r_\fa(n)=r_{\fa^{-1}}(n)$. We define $R_\fa(n)$ to be the 
number of proper, integral $\co$-ideals of norm $n$ in the $\co$-genus
of $\fa$; that is, such that the image in $\Pic(\co)/\Pic(\co)^2$
agrees with the image of $\fa$.
For any integer $k$ we set
$$\delta(k)=2^{\mathrm{\#\{prime\ divisors\ of\ }(k,D)\}}.$$

The reciprocity map of class field theory is always normalized in
the arithmetic fashion.


\section{Preliminaries on elliptic curves}


\subsection{CM points, Heegner diagrams, and Serre's construction}
\label{CM}

Let $S$ be an $\co_K$-scheme and  let $\co=\co[c]\subset \co_K$ be the
order of conductor $c$.  Assume $(c,N)=1$.
An elliptic curve $E\map{}S$ is said to have CM by $\co$
if there is an embedding $\co\hookrightarrow \End_S(E)$.  We always 
assume that such an embedding is normalized, in the sense that the
action of $\co$ on the pull-back of the tangent sheaf of $E$
by the identity section agrees with the action given by viewing
the structure sheaf of $S$ as a sheaf of $\co$-algebras.
We say that $\co$ is the CM-order of $E$, or that $E$ has CM by exactly
$\co$, if this action does not extend to any larger order.  
A \emph{Heegner diagram} of conductor $c$ over
$S$, $h$, is an $\co$-linear cyclic $N$-isogeny of elliptic curves
$h:E\map{}E'$ over $S$, such that $E$ and $E'$ both have CM by
exactly $\co$.  An isogeny of Heegner diagrams means an isogeny of the 
underlying $\Gamma_0(N)$-structure; i.e. a commutative diagram
$$\xymatrix{
E_0\ar[d]_f\ar[r]^{h_0}&E_0'\ar[d]^{f'}\\
E_1\ar[r]^{h_1}&E_1'
}$$
in which the vertical arrows are isogenies of elliptic curves over $S$,
and the map $f$ takes the scheme-theoretic kernel of $h_0$
isomorphically to the scheme-theoretic kernel of $h_1$.
The degree of such an isogeny is defined to be the degree of $f$, which 
is also the degree of $f'$.
Any Heegner diagram $h$ over $S$ gives rise to an $S$-valued point
of $X_0(N)_{/\Z}$, which we also denote by $h$.  Since
$X_0(N)$ is not a fine moduli space, Heegner diagrams which are not
isomorphic over $S$ may give rise to the same $S$-valued point on $X_0(N)$.

If $E_{/S}$ is an elliptic curve with CM by $\co$ and
$\fa$ is a proper fractional $\co$-ideal, a theorem of Serre 
\cite[Theorem 7.2]{conrad}
guarantees that the functor from $S$-schemes to $\co$-modules
$T\mapsto E(T)\otimes_\co \fa$ is represented
by an elliptic curve which we denote by $E\otimes_\co \fa$.
Define $E^\fa=E\otimes_\co\fa^{-1}$.  
As in \cite[Corollary 7.11]{conrad}, this construction 
extends to Heegner diagrams, and so to any Heegner 
diagram $h:E\map{}E'$ of conductor $c$
over $S$ and any $\fa$ as above,
we obtain  a new Heegner diagram
$$h^\fa:E^\fa\map{} E'^\fa.$$

If $S=\Spec(\C)$ and $E$ is an elliptic 
curve over $S$ with CM by exactly $\co$, then $E(\C)\iso \C/\fb$ for some 
proper fractional $\co$-ideal $\fb$, and we have an analytic isomorphism
$E^\fa(\C)\iso \C/\fa^{-1}\fb.$
By the Main Theorem of Complex Multiplication, the right hand side 
is analytically isomorphic to $E^\sigma(\C)$ for any 
$\sigma\in\mathrm{Aut}(\C/K)$ whose restriction to $H[c]$ (the ring
class field of conductor $c$) agrees with $\fa$
under the Artin map $\Pic(\co)\iso\Gal(H[c]/K)$.
In particular $E$ has a model over $H[c]$,  $E^\sigma$ and $E^\fa$
are isomorphic over $\C$,
and $\Gal(H[c]/K)$ acts transitively on the $\C$-isomorphism classes
of elliptic curves over $H[c]$ with CM by exactly $\co$.
Similarly all Heegner diagrams over $\C$ of conductor $c$
have models over the ring class field of conductor $c$.  If $h$ is
a Heegner diagram of conductor $c$ defined over $H[c]$, we define the
\emph{orientation} of $h$ to be the annihilator in $\co$ of the
kernel of $h:E(\C)\map{}E'(\C)$.
It is an ideal $\mathcal{N}$ of $\co$ such that $\co/\mathcal{N}\iso\Z/N\Z$.
Then $\Gal(H[c]/K)$ acts transitively on the $\C$-isomorphism classes
of conductor $c$ Heegner points with a given orientation.


\subsection{Hecke action on CM points}
\label{Hecke action}


Let $\mathfrak{L}$ denote the set of lattices in $K$, modulo multiplication
by $K^\times$.  The $K^\times$-class of a lattice $L$ will be denoted $[L]$.
For any $[L]\in\mathfrak{L}$ we define the 
conductor of $[L]$ to be the conductor of the left order of $L$; that is,
the conductor of the order 
$\co(L)=\{\alpha\in K\mid \alpha L\subset L\}$.  
Every lattice of conductor $c$ is represented uniquely (up to $K^\times$
action) by an element of $\Pic(\co)$, where $\co\subset K$ is the 
order of conductor $c$.

We have the usual action of Hecke operators $\{T_m\}$
on formal sums of classes in $\mathfrak{L}$, which we wish to make
explicit.  The following lemma is an elementary exercise.  

\begin{Lem}\label{counting lattices}
Suppose we are given orders $\co$ and $\co'$ of $K$
of conductors $c$ and $d$, respectively,
and a proper fractional $\co$-ideal $\mathfrak{c}$ 
(resp. $\co'$-ideal $\mathfrak{d}$).
If $c|d$ then the multiplicity of $[\mathfrak{c}]$ in the formal sum
$T_m[\mathfrak{d}]$ is equal to 
$r_{\mathfrak{c}\mathfrak{d}^{-1}\co}(mc/d)$.
If instead $d|c$, then the multiplicity of $[\mathfrak{c}]$ in 
$T_m[\mathfrak{d}]$ is given by 
$|\co'^\times||\co^\times|^{-1}
r_{\mathfrak{c}\mathfrak{d}^{-1}\co'}(md/c).$
\end{Lem}

\begin{Lem}[Euler system relations]
With notation as in the introduction and $2u=|\co_K^\times|$,
$$
T_{p^{r}}(h_s)=\Norm_{H_{s+r}/H_s}(h_{s+r})+T_{p^{r-1}}(h_{s-1})$$
if $r,s>0$, and 
$$
T_p(h_0)=
\left\{\begin{array}{ll}
u\cdot \Norm_{H_1/H_0}(h_1)+ (\sigma_p+\sigma_p^*)h_0
& \mathrm{if\ } \epsilon(p)=1\\
u\cdot \Norm_{H_1/H_0}(h_1)& \mathrm{if\ } \epsilon(p)=-1.
\end{array}\right.
$$
\end{Lem}

\begin{proof}
We give a brief sketch of the proof of the first relation.  Let
$\mathfrak{d}$ be a proper $\co_{s+r}$-ideal such that
$\C/\mathfrak{d}\iso E_{s+r}(\C)$, and for any $0\le t\le s+r$,
set $\mathfrak{d}_t=\mathfrak{d}\co_t$, so that 
$E_t(\C)\iso\C/\mathfrak{d}_t$.  By the theory of complex multiplication,
the complex elliptic curves underlying the $\Gamma_0(N)$-structures appearing
in the divisor  $\Norm_{H_{s+r}/H_s}(h_{s+r})$ are exactly the complex 
tori of the form $\C/\mathfrak{d}'$ where $\mathfrak{d}'$ is a 
proper $\co_{s+t}$-ideal satisfying $\mathfrak{d}'\co_s=\mathfrak{d}_s$.
Using Lemma \ref{counting lattices}, such a $\mathfrak{d}'$
occurs exactly once in the formal sum
$T_{p^r}[\mathfrak{d}_s]$, and does not occur in 
$T_{p^{r-1}}[\mathfrak{d}_{s-1}]$.  As the formal sum of lattices
$T_{p^r}[\mathfrak{d}_s]-T_{p^{r-1}}[\mathfrak{d}_{s-1}]$ has degree
$p^r$, it must be exactly the formal sum of $[\mathfrak{d}']$
with $\mathfrak{d}'$ as above.
\end{proof}


\subsection{The Serre-Tate theorem}


We recall the Serre-Tate theory of deformations of elliptic curves.
More detail can be found in  \cite[\S 3]{conrad} and \cite[Chapter 6]{Goren}. 
Let $k$ be a field of nonzero
characteristic $\ell$ and define $\mathcal{C}_k$ to be the 
category of local Artinian 
algebras $(R,\mathfrak{m}_R)$ with residue field $k$,
together with a chosen isomorphism
$R/\mathfrak{m}_R\iso k$, with morphisms given by local algebra
maps inducing the identity on $k$.  Given an elliptic curve 
$E\map{}\Spec(k)$, and some $R\in\mathcal{C}_k$, 
we define a \emph{deformation} of 
$E$ to $R$ to be an elliptic curve $E_R\map{}\Spec(R)$ together 
with an isomorphism between the closed fiber of $E_R$ and $E$.  
Similarly, we may define the notion of
a deformation of the $\ell$-divisible group of an elliptic curve over $k$. 
For $(R,\mathfrak{m}_R)$ an object of $\mathcal{C}_k$,
let $\mathrm{DEF}_{R}$ denote the
category of pairs $(E,G)$ where $E$ is an elliptic curve over $k$
and $G$ is a deformation to $R$ of the $\ell$-divisible group of $E$.
A morphism from $(E,G)$ to $(E',G')$ is a pair $(f,\phi)$
where $f:E\map{}E'$ is a morphism of elliptic curves over $\Spec(k)$
and  $\phi:G\map{}G'$ is a map of $\ell$-divisible groups such that the
base change of $\phi$ to the closed fiber is the map on $\ell$-divisible
groups over $\Spec(k)$ induced by $f$.

\begin{Thm}[Serre-Tate]
For any obect $(R,\mathfrak{m}_R)$ of $\mathcal{C}_k$, the functor 
from elliptic curves over $R$ to $\mathrm{DEF}_{R}$ which 
sends $E$ to the pair $(E\times_R k, E[\ell^\infty])$
is an equivalence of categories, where $E[\ell^\infty]$ 
denotes the $\ell$-divisible group of $E$.
\end{Thm}

Now assume that $k$ is algebraically closed
and fix an ordinary elliptic curve $E$ over $k$.  We have
$E[\ell^\infty]\iso \mu_{\ell^\infty}\oplus \Q_\ell/\Z_\ell$
as $\ell$-divisible groups over $k$.
For any $R\in\mathcal{C}_R$ there is a distinguished deformation 
of the $\ell$-divisible group of $E$
to  an $\ell$-divisible group over $R$, namely the deformation
$\mu_{\ell^\infty}\oplus \Q_\ell/\Z_\ell$.
Applying the Serre-Tate theorem,
we obtain an elliptic curve over $R$ called the \emph{Serre-Tate canonical 
lift} of $E$ to $R$.

As explained in \cite[\S 3]{conrad}, a theorem of Grothendieck allows 
one to replace ``local Artinian'' by ``complete local Noetherian'' in the 
definition of  $\mathcal{C}_k$,
and the discussion above holds verbatim.


\section{The $p$-adic $L$-function}
\label{L function}


In this section we quickly recall the essential properties of 
Hida's $p$-adic $L$-function $\cL_f$ and Perrin-Riou's calculation
of its linear term.  
We refer the reader to \cite{Hida, nekovar, pr1} for more detailed treatments.
Assume that $D$ is odd. Recall that $\cA_0\subset\Q^\alg$ 
is the ring of integers of
a number field with closure $\cA$ in $\Q_p^\alg$, $\cB$ is the fraction
field of $\cA$, and $\alpha\in\cA^\times$ is the unit root of 
$X^2-a_p(f)X+p$.

Set 
$$
M_2(\Gamma_0(Np^k),\cA)=M_2(\Gamma_0(Np^k),\cA_0)\otimes_{\cA_0}\cA
$$
and let $M_2(\Gamma_0(Np^\infty),\cA)$ be the completion of
$\cup_k M_2(\Gamma_0(Np^k),\cA)$ with respect to the $p$-adic
supremum norm on Fourier coefficients.
To any $s\ge0$, $\sigma\in\Gal(H_s/K)$, and integer $C$ prime to
$Dp$, Perrin-Riou \cite[\S2.2.3]{pr1}   associates 
a measure $\Phi_\sigma^C$ on $\Z_p^\times$ with 
values in the space $M_2(\Gamma_0(Np^\infty),\cA)$.
These are compatible as $s$ and $\sigma$ vary in the following sense:
there is a measure $\Phi^C$ on $\Gal(H_\infty/K)\times\Z_p^\times$
with values in $M_2(\Gamma_0(Np^\infty),\cA)$
such that for any continuous characters 
$$
\eta:\Gal(H_\infty/K)\map{}\Q_p^{\alg,\times}
\hspace{1cm}
\psi:\Z_p^\times\map{}\Q_p^{\alg,\times}
$$
such that $\eta$ factors through $\Gal(H_s/K)$ we have the relation
$$\int_{\Gal(H_\infty/K)\times\Z_p^\times}\eta\psi\ d\Phi^C
=\sum_{\sigma\in\Gal(H_s/K)}\eta(\sigma)\int_{\Z_p^\times}
\psi \ d\Phi^C_\sigma
$$
in $M_2(\Gamma_0(Np^\infty),\cA)\otimes_\cA\Q_p^\alg$.   

Use the notation $\tilde{T}_\ell$ to denote Hecke operators
acting on modular forms of level $\Gamma_0(Np^\infty)$, 
to distinguish them from the operators on level
$\Gamma_0(N)$. Define Hida's ordinary projector \cite[\S 7.2]{Hida book}
$$e^\ord:M_2(\Gamma_0(Np^\infty),\cA)\map{}M_2(\Gamma_0(Np),\cA)$$
by $e^\ord(g)=\lim_{k\to\infty}U^{k!}(g)$, where $U=\tilde{T}_p$ is given
by $U(\sum a_nq^n)=\sum a_{np}q^n$ and the limit is with respect to 
the supremum norm on Fourier coefficients.
Define modular forms of level $\Gamma_0(Np)$ by 
$$f_0(z)=f(z)-\frac{p}{\alpha}f(pz)
\hspace{1cm}
f_1(z)=f(z)-\alpha f(pz).
$$
These are eigenforms for \emph{all} Hecke operators $\tilde{T}_\ell$, and
satisfy $a_\ell(f_0)=a_\ell(f)=a_\ell(f_1)$ if $\ell\not=p$, and
$a_p(f_0)=\alpha$, $a_p(f_1)=p/\alpha$.
The $\cB$-algebra
generated by the Hecke operators $\tilde{T}_\ell$ with $(\ell,Np)=1$
acting on $M_2(\Gamma_0(Np),\cA)\otimes_\cA\cB$ is semi-simple, and so
contains an idempotent $e_f$ such that
$e_f\circ \tilde{T}_\ell= a_\ell(f)e_f$. By \cite[\S4]{Hida}
there is an idempotent $e_{f_0}$ in the algebra generated by \emph{all} Hecke
operators $\tilde{T}_\ell$,
such that $e_{f_0}\circ \tilde{T}_\ell=a_\ell(f_0)e_{f_0}$ for
every $\ell$.  As operators on modular forms, 
$e_{f_0}=e_{f_0}e_f$.
Define a linear functional 
$$l_f:M_2(\Gamma_0(Np^\infty),\cA)\otimes_\cA\cB\map{}\cB$$
by $l_f(g)=a_1( e_{f_0}e^\ord g)$, and set 
$L_f=(1-p/\alpha^2)(1-1/\alpha^2)l_f$ 
(this is denoted $\tilde{L}_{f_0}$ in \cite{pr1}).

\begin{Lem}\label{linear lemma}
The linear functional 
$L_f:M_2(\Gamma_0(Np^\infty),\cA)\otimes_\cA\cB\map{}\cB$ satisfies
\begin{enumerate}
\item $L_f=L_f\circ e^\ord$,
\item $L_f(f)=1-1/\alpha^2$,
\item if $g\in M_2(\Gamma_0(Np^\infty),\cA)$
is such that $a_m(g)=0$ for all $(m,N)=1$, then $L_f(g)=0$,
\item for any positive integer $m$, 
$L_f\circ \tilde{T}_m=a_m(f_0) L_f$.  In particular,
$L_f\circ U=\alpha L_f$.
\end{enumerate}
\end{Lem}

\begin{proof}
The first claim is trivial, since $e^\ord\circ e^\ord=e^\ord$.  
The second follows from $l_f(f_0)=1$,
$l_f(f_1)=0$.  If $g$ satisfies $a_m(g)=0$ for all $(m,N)=1$,
then so does $e_f e^\ord g$, 
so we may assume that $g$ has level $\Gamma_0(Np)$ and that 
$\tilde{T}_\ell g=a_\ell(f)g$ for $(\ell,Np)=1$.
By Atkin-Lehner theory, $g$ is a linear combination of $f_0$ and $f_1$.
Since $a_1(g)=0$, $g$ must be a scalar multiple of $f_0-f_1$.
But $a_p(f_0-f_1)\not=0$, so this scalar must be $0$.
The final claim follows from 
$e_{f_0}\circ\tilde{T}_m=a_m(f_0) e_{f_0}.$
\end{proof}

\begin{Rem}
Contrary to the proof of \cite[Proposition II.5.10]{nekovar},
the weaker hypothesis that $a_m(g)=0$ for all $(m,Np)=1$ is not 
sufficient to conclude that  $L_f(g)=0$.
The modular form $g=f_0-f_1$ provides a counterexample.  
\end{Rem}

Whenever $\psi$ is a continuous character of $\Gamma$, we extend
$\psi$ to a character of $\Z_p^\times$ using the usual projection
$\langle\ \rangle:\Z_p^\times\map{}\Gamma$.
We now define the $p$-adic $L$-function $\cL_f$ of the introduction
(compare \cite[D\'efinition 2.4]{pr1}, but note that Perrin-Riou's
$\psi(C)=\psi(\mathrm{Frob}_{C\co_K})$ is our
$\psi(C)^2$).
For any  continuous character $\eta\cdot \psi$ of 
$\Gal(H_\infty/K)\times\Gamma$, set
$$
\cL_f(\eta,\psi)=
\frac{1}{1-C\epsilon(C)\psi(C)^{-2}}\cdot 
L_f\left(\int_{\Gal(H_\infty/K)\times\Z_p^\times}\eta\cdot \psi \
d\Phi^C\right),
$$
where $C$ is chosen so that 
$(1-C\epsilon(C)\langle C\rangle^{-2})\in\Z_p[[\Gamma]]^\times$.
The resulting $\cL_f\in\cA[[\Gal(H_\infty/K)\times\Gamma]]\otimes_{\cA}\cB$ 
does not depend on the choice of $C$.
Any finite order character $\eta\cdot \psi$
of $\Gal(H_\infty/K)\times\Gamma$ determines a character 
$$
\chi(\fb)=\eta(\mathrm{Frob}_\fb)\cdot \psi(\N(\fb))
$$
on ideals of $\co_K$ prime to $p$, and there is an interpolation formula
\cite[Th\'eor\`eme 1.1]{pr1} relating
$\cL_f(\eta,\psi)$ to $L(f,\bar{\chi},1)$, where
$L(f,\bar{\chi},s)$ is the Rankin product of the $L$-function
of $f$ and the $L$-function of the theta series associated to $\bar{\chi}$.

\begin{Prop}\label{L prop}
Let $\mathbf{1}$ denote the trivial character of $\Gamma$.
Then $\cL_f(\eta,\mathbf{1})=0$ for all continuous characters
$\eta$ of $\Gal(H_\infty/K)$.
Furthermore, in the notation of (\ref{two variable decomposition}),
$\cL_{f,0}=0$ and 
$$\log_p(\gamma_0)\cdot \cL_{f,1}(\eta)=
\sum_{\sigma\in\Gal(H_s/K)} \eta(\sigma) L_f(G_\sigma)$$
for every character $\eta$ of $\Gal(H_s/K)$,
where $G_\sigma\in M_2(\Gamma_0(Np^\infty),\cA)$
is defined by 
$$
G_\sigma= \frac{1}{1-C\epsilon(C)}\cdot 
\int_{\Z_p^\times}\log_p\ d\Phi_\sigma^C.
$$
\end{Prop}

\begin{proof}
Fix an integer $s>0$.  For each $\sigma\in\Gal(H_s/K)$
define 
$$
\cL^\sigma(\psi)=\frac{1}{1-C\epsilon(C)\psi(C)^{-2}}\cdot 
\int_{\Z_p^\times}\psi\ d\Phi_\sigma^C\ \ 
\in M_2(\Gamma_0(Np^\infty),\cA),
$$
a function on continuous characters $\psi$ of $\Gamma$ with the property
that
$$
\cL_f(\eta,\psi)= \sum_{\sigma\in\Gal(H_s/K)}\eta(\sigma)
L_f(\cL^\sigma(\psi))
$$
for any $\psi$ and any character $\eta$ of $\Gal(H_s/K)$.
By \cite[Remarque 3.19]{pr1} $a_m(\cL^\sigma(\mathbf{1}))=0$
whenever $p\mid m$, and so $U\cL^\sigma(\mathbf{1})=0$.
Lemma \ref{linear lemma}(d) now implies $L_f(\cL^\sigma(\mathbf{1}))=0$.
Since $s$ and $\eta$ were arbitrary, we deduce
$\cL_f(\eta,\mathbf{1})=0$ for all finite order $\eta$, 
hence for all continuous $\eta$ (since $\cL_f(\ ,\mathbf{1})\in
\cA[[\Gal(H_\infty/K)]]\otimes_{\cA}\cB$).  This
is equivalent to $\cL_{f,0}=0$.
Finally, recall that $\langle\ \rangle$ denotes the projection 
$\Z_p^\times\map{}\Gamma$ and compute
\begin{eqnarray*}
\lim_{t\to 0} \frac{\cL_f(\eta,\langle\ \rangle^t)}{t}
& = & \sum_{\sigma\in\Gal(H_s/K)} \frac{d}{dt}\left[ 
\frac{\eta(\sigma)}{1-C\epsilon(C)\langle C\rangle^{-2t}}\cdot 
L_f\left(\int_{\Z_p^\times}\langle x\rangle^t\ d\Phi_\sigma^C(x)\right)
\right]_{t=0}\\
& = & \sum_{\sigma\in\Gal(H_s/K)} \ 
\frac{\eta(\sigma)}{1-C\epsilon(C)}\cdot 
\frac{d}{dt}\left[L_f\left(\int_{\Z_p^\times}
\langle x\rangle^t d\Phi_\sigma^C(x)\right)\right]_{t=0}\\
\end{eqnarray*}
where in the second equality we have used the fact, proved above,
that $L_f\left(\int_{\Z_p^\times}\mathbf{1}\ d\Phi_\sigma^C\right)=0$.
Differentiating under the integral and using
$\log_p(\gamma_0) \cL_{f,1}(\eta)
= \lim_{t\to 0} \frac{1}{t}\cL_f(\eta,\langle\ \rangle^t)$
proves the claim.
\end{proof}

Fix $s\ge 0$ and $\sigma\in\Gal(H_s/K)$.  Choose a proper integral
$\co_s$-ideal, $\fa$, such that the class of $\fa$ in $\Pic(\co_s)$
corresponds to $\sigma$ under the Artin symbol.
For any positive integer $n$ prime to $p$ and any positive divisor
$d|n$, define
$$
\epsilon_\fa(n,d)=\left\{\begin{array}{ll}
\left(\frac{D_1}{d}\right)\left(\frac{D_2}{-Nn/d}\right)
\chi_{D_1,D_2}(\fa\co_K)& \mathrm{if\ } \gcd(d,n/d,D)=1\\
0 & \mathrm{otherwise}
\end{array}\right.
$$
where $D=D_1D_2$ is the factorization into fundamental discriminants
with $(d,D)=|D_2|$ and $\chi_{D_1,D_2}$ is the associated genus character.
That is, the quadratic character of $\Pic(\co_K)$ associated to
the extension $K(\sqrt{D_1})=K(\sqrt{D_2})$.
Set 
$$
\sigma'_\fa(n)=\sum_{ \stack{d|n}{d>0} }\epsilon_\fa(n,d)\log_p(n/d^2).
$$

\begin{Prop}(Perrin-Riou)\label{pr calculation}
For any positive integer $m$ divisible by $p$, the $m^\mathrm{th}$
Fourier coefficient of $G_\sigma$ is given by
$$a_m(G_{\sigma})=-\sum_{\stack{n>0}{(n,p)=1}} 
r_{\fa\mathfrak{d}_s}(m|D|-nN)\sigma'_\fa(n)$$
where $\mathfrak{d}_s=(\sqrt{D}\co_K)\cap\co_s$.
\end{Prop}
\begin{proof}
This is \cite[Proposition 3.18]{pr1}, where
$G_\sigma$ is denoted $L'_{p,\sigma,\langle\ \rangle}$.
The missing minus sign in the
statement of Perrin-Riou's Proposition 3.18 is a 
typographical error, as the proof makes clear.

In Perrin-Riou's statement $r_{\fa\mathfrak{d}_s}$
appears as $r_{\fa'}$ where $\fa'=\mathfrak{D}\fa$, and  (p. 484)
``$\mathfrak{D}$ est le $\co_s$-id\'eal engendr\'e par $\sqrt{D}$''.
That is, $\mathfrak{D}=\sqrt{D}\co_s\not=\mathfrak{d}_s$.
Later, on p. 486, Perrin-Riou writes ``Lorsque $s=0$,
$\fa'$ et $\fa$ sont \'equivalent'', although under the stated
definition of $\mathfrak{D}$ they are equivalent even when $s\not=0$,
suggesting that an unannounced change of notation has occured.
The formulas of \cite[\S 3.2.3]{pr1}
are correct with $\mathfrak{D}$ defined as above, while those of 
\cite[\S 3.3]{pr1} are correct with $\mathfrak{D}$ replaced by our
$\mathfrak{d}_s$.  Especially, in the proof of \cite[Lemme 3.17]{pr1}
one must interpret $\mathfrak{D}$ as our $\mathfrak{d}_s$ in order to pass
from equation (3.7) to (3.8) 
(``On remplace ensuite $n$ par $\delta_2 n$...'').  
The key point is
$$
r_{\mathfrak{D}_1^{-1}\fa}(m\delta_1-nN)
=
r_{\mathfrak{D}^{-1}_1\mathfrak{D}_2\fa}(m\delta-n\delta_2 N)
$$
in which $\delta=|D|=\delta_1\delta_2$ and $\mathfrak{D}_i$ is
the $\co_s$-ideal of norm $\delta_i$ (the equality is
seen by using the map on $\co_s$-ideals $\fb\mapsto \mathfrak{D}_2\fb$
to identify the sets of ideals being counted).
Using $\mathfrak{D}_1^{-1}\mathfrak{D}_2=\mathfrak{d}_s$
in $\Pic(\co_s)$, one obtains the correct formula.
Also, the first displayed equation in the proof of 
\cite[Lemme 3.17]{pr1} appears to be in error;
the two $p$-adic modular forms in the second equality differ by 
shifting Fourier coefficients by $\delta_1$ 
(see \cite{pr1} (2.4) and Lemme 3.1).  
This misstatement has no effect on the proof.

Perrin-Riou's $\fa$ is our $\fa^{-1}$, but both $r_{\fa\mathfrak{d}_s}$ 
and $\sigma'_\fa$
are unchanged by $\fa\mapsto\fa^{-1}$. For $\sigma'_\fa$ this is obvious;
for $r_{\fa\mathfrak{d}_s}$ use the fact that inversion 
agrees with complex conjugation
in $\Pic(\co_s)$, the fact that complex conjugation preserves norms,
and the fact that $\mathfrak{d}_s$ has order two in $\Pic(\co_s)$.
\end{proof}

\begin{Lem}\label{genus characters}
Suppose that $n$ is prime to $p$ and that there exists a 
proper integral $\co_s$-ideal $\fb$ in the $\Pic(\co_s)$-class
of $\fa$ with $\N(\fb)\equiv -nN\pmod{Dp}$.  Then
$$\sigma'_\fa(n)=\sum_{\ell|n}\log_p(\ell)\cdot
\left\{\begin{array}{ll}
0 & \mathrm{if\ }\epsilon(\ell)=1\\
\ord_\ell(\ell n)\delta(n)R_{\fa\fn\mathfrak{c}}(n/\ell)
& \mathrm{if\ }\epsilon(\ell)=-1\\
\ord_\ell(n)\delta(n)R_{\fa\fn\mathfrak{c}}(n/\ell) 
& \mathrm{if\ }\epsilon(\ell)=0
\end{array}\right.$$
where in the second and third cases $\fn$ is any integral $\co_s$-ideal
of norm $N$ and  $\mathfrak{c}$ is any proper integral
$\co_s$-ideal with $\N(\mathfrak{c})\equiv-\ell\pmod{Dp}$.
\end{Lem}

\begin{proof}
By \cite[Proposition IV.4.6(b)]{gz}, the stated equality holds with 
$R_{\fa\fn\mathfrak{c}}(n/\ell)$ 
replaced by $R_{\fa\fn\mathfrak{c}\co_K}(n/\ell)$;
that is, if we count integral $\co_K$-ideals of norm $n/\ell$ in the
$\co_K$-genus of $\fa\fn\mathfrak{c}\co_K$.
So, we only need show that 
$R_\mathfrak{anc}(n/\ell)=R_{\mathfrak{anc}\co_K}(n/\ell)$
under the stated hypotheses.
The map $I\mapsto I\co_K$ takes the collection 
$\mathfrak{R}_{\mathfrak{anc}}(n/\ell)$ of proper 
$\co_s$-ideals of norm $n/\ell$ in the $\co_s$-genus of $\mathfrak{anc}$
injectively to the set $\mathfrak{R}_{\mathfrak{anc}\co_K}(n/\ell)$
of proper $\co_K$-ideals of norm $n/\ell$
in the $\co_K$-genus of $\mathfrak{anc}\co_K$.  It suffices to show
that this map has an inverse.  More precisely, we show
that the map $J\mapsto J\cap\co_s$ from integral $\co_K$-ideals
of norm prime to $p$  to integral $\co_s$-ideals of norm prime to $p$
restricts to a map $\mathcal{R}_{\mathfrak{anc}\co_K}(n/\ell)\map{}
\mathcal{R}_{\mathfrak{anc}}(n/\ell)$.

Suppose $I=J\cap\co_s$ is an integral $\co_s$-ideal of norm $n/\ell$ such that 
$J\in\mathfrak{R}_{\mathfrak{anc}\co_K}(n/\ell)$. 
Set $p^*=(-1)^\frac{p-1}{2}p$.
Genus theory (for example, \cite[\S 6.A]{cox} discusses the genus 
theory of $\co_K$ at length, and that of $\co_s$ is similar) 
gives a canonical isomorphism
$$
\Pic(\co_s)/\Pic(\co_s)^2\iso \Pic(\co_K)/\Pic(\co_K)^2\times
\Gal(K(\sqrt{p^*})/K)
$$
under which the $\co_s$-genus of $I$ is sent to the
$\co_K$-genus of $J=I\co_K$
in the first factor, and to its Artin symbol 
$\left(\frac{I}{K(\sqrt{p^*})/K}\right)=
\left(\frac{\N(I)}{\Q(\sqrt{p^*})/\Q}\right)$
in the second factor. 
The same holds with $I$ replaced by $\mathfrak{bnc}$, and 
since the $\co_K$-genera of $J$ and $\mathfrak{bnc}\co_K$ agree
by assumption, $I\in \mathfrak{R}_{\mathfrak{anc}}(n/\ell)=
\mathfrak{R}_{\mathfrak{bnc}}(n/\ell)$
if and only if 
$$
\left(\frac{\N(I)}{\Q(\sqrt{p^*})/\Q}\right)=
\left(\frac{\N(\mathfrak{bnc})}{\Q(\sqrt{p^*})/\Q}\right)
$$
which occurs if and only if 
$\left(\frac{\N(I)}{p}\right)
=\left(\frac{\N(\mathfrak{bnc})}{p}\right).$
Since $\N(I)=n/\ell$ and  $\N(\mathfrak{bnc})\equiv nN^2\ell \pmod{p}$
we are done.
\end{proof}

\begin{Cor}\label{pr calc 2}
Let $\kappa\in\Gal(H_s/K)$ be the Artin symbol of $\mathfrak{d}_s$.
For any positive integer $m$ divisible by $p$, the $m^\mathrm{th}$
Fourier coefficient of $G_{\sigma\kappa}$ is given by the expression
$$
-\sum_{\stack{n>0}{(n,p)=1}}\sum_{\ell|n}
\log_p(\ell)\cdot r_\fa(m|D|-nN)
\cdot\left\{\begin{array}{ll}
0 & \mathrm{if\ }\epsilon(\ell)=1\\
\ord_\ell(\ell n)\delta(n)R_{\fa\fn\mathfrak{c}}(n/\ell)
& \mathrm{if\ }\epsilon(\ell)=-1\\
\ord_\ell(n)\delta(n)R_{\fa\fn\mathfrak{c}}(n/\ell) 
& \mathrm{if\ }\epsilon(\ell)=0
\end{array}\right.$$
where in the second and third cases $\fn$ is any integral $\co_s$-ideal
of norm $N$ and  $\mathfrak{c}$ is any proper integral
$\co_s$-ideal with $\N(\mathfrak{c})\equiv-\ell\pmod{Dp}$.
\end{Cor}

\begin{proof}
Combine Proposition \ref{pr calculation} and Lemma \ref{genus characters},
and use $\sigma_{\fa}'=\sigma_{\fa\mathfrak{d}_s}'$ (which 
follows from the definition of $\sigma'$ and the fact that
$\mathfrak{d}_s\co_K$ is principal) and $\kappa^2=1$.
\end{proof}


\section{The $p$-adic height pairing}
\label{height pairing}


In this section we recall some known facts about $p$-adic N\'eron 
symbols and  $p$-adic height
pairings on abelian varieties and, when the abelian variety is the 
Jacobian of a curve, the connection with $p$-adic 
N\'eron symbols and intersection theory on the curve.  


\subsection{Intersection theory}
\label{intersection theory}


Let $R$ be complete DVR, $S=\Spec(R)$.
Let  $\uX\map{}S$ be an  integral, proper scheme over $S$ with
generic fiber a smooth curve $X$, and
suppose $\uC$ and $\uD$ are effective Cartier  
divisors with no common components.
Define the intersection multiplicity
$i_y(\uC,\uD)$ at a closed point $y$ of $\uX$ to be the length of the 
$\co(\uX)_y$-module $\co(\uX)_y/(f,g)$ where $f$ and $g$ are
defining equations of $\uC$ and $\uD$ in a neighborhood of $y$.
Define the total intersection multiplicity
$i(\uC,\uD)=\sum_y i_y(\uC,\uD)[k(y):k(s)]$ where $s$ is the closed
point of $S$ and the sum is over closed points of $\uX$

We now assume that $\uX$ is regular (in particular we need not
distinguish between Weil divisors
and Cartier divisors), and record some fundamental
properties of the total intersection multiplicity.  We refer the  reader to
\cite{gross} and \cite[Chapter III]{Lang} for details.
The total intersection multiplicity  
is bi-additive, and so extends to divisors with rational
coefficients.  
We define, for $C$ and $D$ 
degree zero divisors on $X$ with disjoint support, 
$$
[C,D]= i(\uC+C',\uD)=i(\uC,\uD+D')
$$
where $\uC$ and $\uD$ are the horizontal divisors on $\uX$
whose generic fibers are $C$ and $D$, respectively,
and $C'$ (resp. $D'$) is a fibral divisor with rational coefficients
chosen so that the symbol $i(\uC+C',\ )$ (resp. $i(\ ,\uD+D')$) 
vanishes on all fibral divisors.  Let $L$ be the fraction field of 
$R$ and let $v$ denote the normalized valuation on $L$, so that
$v(\varpi)=1$ for a uniformizer $\pi$. If $C=(f)$ is a principal divisor
then $[C,D]=v(f(D))$ where $D=\sum n_i (D_i)$ is a linear combination
of prime divisors $D_i$ with residue field $L_i$ and
\begin{equation}\label{divisor decomp}
f(D)=\prod_i \N_{L_i/L} (f(D_i)^{n_i}).
\end{equation}


\subsection{$p$-adic N\'eron symbols I}
\label{local symbols}


We now define local $p$-adic N\'eron symbols on abelian varieties.  
The contents
of this subsection are taken from \cite[\S 4]{pr1} essentially verbatim.

Let $\ell$ be a rational prime
and $L$ a finite extension of $\Q_\ell$. Let $A$
be an abelian variety over $L$ and assume that either $\ell\not=p$
or that $A$ has good reduction.  Fix a nontrivial continuous 
additive character $\rho:L^\times\map{}\Z_p$. 
If $\ell=p$ we assume that $\rho$ is ramified.

\begin{Prop}
\label{pr height}
There is a $\Q_p$-valued N\'eron symbol
$\langle\mathfrak{C}, d\rangle= \langle \mathfrak{C}, d\rangle_{A,\rho}$
defined whenever $\mathfrak{C}$ is an algebraically trivial
divisor on $A$, $d$ is a zero cycle of degree zero on 
$A$ rational point-by-point over $L$,
and the supports of $\mathfrak{C}$ and $d$ have no common points.
This symbol satisfies
\begin{enumerate}
\item $\langle\ ,\ \rangle$ is bilinear (whenever this makes sense)
and invariant under translation by elements of $A(L)$, 
\item if $\mathfrak{C}=(h)$ is principal then 
$\langle \mathfrak{C},d\rangle=\rho(h(d))$,
where $h(d)=\prod_i f(d_i)$ is defined as in (\ref{divisor decomp}),
\item for any endomorphism $\phi:A\map{}A$, 
$\langle \phi^*\mathfrak{C}, d\rangle=\langle \mathfrak{C},\phi_*d\rangle$,
\item  for any $x_0\in A(L)$
and any $\mathfrak{C}$ as above, the function $x\mapsto \langle \mathfrak{C}, 
(x)-(x_0)\rangle$ is continuous
for the $\ell$-adic topology on $A(L)$,
\item if $\ell=p$, $L'$ is a finite extension of $L$ contained
in the $\Z_p$-extension of $L$ cut out by $\rho$, 
and $\mathfrak{C}$ is a degree zero divisor on $A_{/L'}$, then
$$
\langle \N_{L'/L} \mathfrak{C}, d\rangle
\subset c^{-1} \rho(\N_{L'/L}(L'))
$$
whenever this is defined, for some constant $c\in\Z_p$ independent of
$L'$, $\mathfrak{C}$, and $d$.
\end{enumerate}
Furthermore, if $\ell\not=p$, or if $\ell=p$ and $A$ has ordinary reduction,
then such a symbol is unique.
\end{Prop}

\begin{proof}
In the case $\ell\not=p$, or $\ell=p$ but $A$ has ordinary reduction,
see the references after \cite[Th\'eor\`eme 4.2]{pr1}) for existence.
In the case $\ell=p$ with non-ordinary reduction, the existence
is \cite[Th\'eor\`eme 4.7]{pr1}.  The translation invariance
is not stated explicitly by Perrin-Riou, but follows from the construction
as in \cite[Lemma 2.14]{Bloch}.
We sketch the proof of the uniqueness. If 
$\langle\ ,\ \rangle'$ is another such 
symbol then we may define 
$$
G(\mathfrak{C},x)=\langle \mathfrak{C},(x)-(0)\rangle -
\langle \mathfrak{C},(x)-(0)\rangle'. 
$$ 
This defines a function $A^\vee(L)\times A(L)\map{}\Q_p$
which is linear in the first variable and continuous in the second.
Using translation invariance and 
the theorem of the square  \cite[Theorem 6.7]{milne}, 
one can show that $G$ is also linear in the second variable.
Hence for fixed $\mathfrak{C}$,
$G(\mathfrak{C},\ )$ defines a continuous linear map
$A(L)\map{}\Q_p$.  If $\ell\not=p$ this map must be 
trivial for topological reasons.
If $\ell=p$ and $A$ has ordinary reduction, then $A^\vee$ also has ordinary
reduction, and \cite[Proposition 4.39]{mazur} implies that the
universal norms from the (ramified) $\Z_p$-extension cut out by $\rho$
have finite index in $A^\vee(L)$.  From this and the boundedness property
(e), we see that $G$ is identically zero.
\end{proof}

When $\ell\not=p$ the N\'eron symbol is compatible with base extension
in the following sense.  If $L'/L$ is a finite extension,
$A'=A\times_L L'$, and $\rho'=\rho\circ\N_{L'/L}$,
then 
\begin{equation}\label{height norms}
\langle \mathfrak{C}, d\rangle_{A',\rho'}=
\langle\N_{L'/L}\mathfrak{C}, d\rangle_{A,\rho}
\end{equation}
for $\mathfrak{C}$ an algebraically trivial divisor on $A'$
and $d$ a point-by-point rational zero cycle of degree zero on $A$.  
This allows us to remove the hypothesis
in Proposition \ref{pr height} 
that $d$ is rational point-by-point, by choosing an extension $L'/L$
over which $d$ becomes pointwise rational and defining
$$
\langle \mathfrak{C}, d\rangle_{A,\rho}= 
[L':L]^{-1}\langle \mathfrak{C}, d\rangle_{A',\rho'}.
$$
This is independent of the choice of $L'$ by (\ref{height norms}).
Property (b) of Proposition \ref{pr height} continues to hold
for this slight extension of the N\'eron symbol, provided
that one extends the definition of $h(d)$ as in (\ref{divisor decomp}).

When $\ell=p$ the N\'eron symbol on $A$ may not uniquely determined
by the properties above, but one
can choose a compatible family (in the sense that (\ref{height norms})
holds) of N\'eron symbols $\langle\ ,\ \rangle_{A',\rho'}$ as 
$L'$ varies over the finite extensions of $L$. Again, this 
allows one to remove the hypothesis that $d$ is defined point
by point.  Perrin-Riou only states the existence of 
compatible families
for subfields of the extension of $L$ cut out by $\rho$, but
the same argument holds for all finite extensions.

\begin{Rem}\label{base compatible}
Although the choice of a N\'eron symbol on $A$ in residue characteristic
$p$ is (sometimes) not unique, our results do not depend on the the choice.
Hence we fix,
once and for all, a choice of N\'eron symbol on $J_0(N)_{H_s,v}$
for every $s$ and every prime $v$ of $H_s$ above $p$, with the understanding
that these choices are compatible as $s$ varies in the sense 
of (\ref{height norms}).
\end{Rem}

Now suppose that $A$ is the Jacobian of a smooth, proper, geometrically
connected curve $X$ over $L$, and that $X$ has an $L$-rational
point $\infty$.  Let $\alpha:X\map{}A$ be the canonical embedding
$x\mapsto (x)-(\infty)$.  Suppose we are given degree zero 
divisors $C$ and $D$ on $X$ with disjoint support.
Pullback by $\alpha$ restricts to an isomorphism 
$\alpha^*:\Pic^0(A)\map{}\Pic^0(X)$, and so 
there is an algebraically trivial divisor $\mathfrak{C}$ whose
associated line bundle pulls back to the line bundle associated to
$C$. Thus $C=\alpha^*\mathfrak{C}+(f)$ for 
some rational function $f$ on 
$X$.  The pair $(\mathfrak{C},f)$ may be chosen so that
$(f)$ is disjoint from $D$
and then it follows that $\mathfrak{C}$ has no points in common
with $\alpha_*D$. We now define
\begin{equation}\label{curve height def}
\langle C, D\rangle_{X,\rho}=\langle \mathfrak{C}, \alpha_*D\rangle_{A,\rho}
+ \rho(f(D)),
\end{equation}
where $f(D)$ is defined by (\ref{divisor decomp}).
This is independent of the choice of $\mathfrak{C}$ (by Proposition
\ref{pr height}(b)) and the choice
of $f$ (which is determined up to $L^\times$ once $\mathfrak{C}$ is chosen).


\subsection{$p$-adic N\'eron symbols II}
\label{global pairing}


Identifying $\Gamma$ with the Galois
group of the unique $\Z_p$-extension of $\Q$ via the cyclotomic character,
the reciprocity
map of class field theory and the $p$-adic logarithm define
an idele class character
$$
\rho_\Q: \mathbf{A}_\Q^\times/\Q^\times\map{}\Gamma\map{\log_p}\Z_p.
$$ 
Fix a finite extension $L/\Q$, let $\rho_L$ be the idele
class character of $L$ defined by $\rho_L=\rho_\Q\circ\N_{L/\Q}$.
Fir each finite place $v$ of $L$, let $\pi_v$ be a uniformizer of
$L_v$ and let $\N(v)$ denote the absolute residue degree of $v$.
We may decompose $\rho_L=\sum_v \rho_{L_v}$ 
as a sum of local characters, and then $\rho_{L_v}(\pi_v)=\log_p(\N(v))$
for any prime $v$ not above $p$.  We note that this does not agree
 with \cite[p. 501]{pr1}, which seems to be in error (note also the remarks
of \cite[\S II.6.4]{nekovar}), although perhaps this is 
attributable to a different
normalization of class  field theory.  We remind the reader that we always
use the \emph{arithmetic} conventions.

Let $A$ be an abelian variety over $L$ with 
good reduction above $p$.  Summing the local N\'eron symbols 
$\langle\ ,\ \rangle_v=\langle\ ,\ \rangle_{A_v,\rho_{L_v}}$
on the completions $A_v=A\times_L L_v$ 
defines a  bilinear pairing on Mordell-Weil groups
\begin{equation}\label{global abelian}
\langle\ ,\ \rangle_{A,L}:A^\vee(L)\times A(L)\map{}\Q_p.
\end{equation}
Indeed, given $a\in A^\vee(L)$ and $b\in A(L)$,
let $\mathfrak{C}$ be an algebraically trivial divisor on 
$A$ which represents $a$ and let $d=\sum n_i(d_i)$ be a zero
cycle of degree zero on $A$ with
$\sum n_id_i=b$.
These can be chosen so that $\mathfrak{C}$
and $d$ have no points in common and we then define 
$$
\langle a,b\rangle_{A,L}=\sum_v \langle 
\mathfrak{C},d\rangle_v
$$
where the sum is over the finite places of $L$.
A different choice of $\mathfrak{C}$ changes the pairing
by 
$$
\sum_v\langle (h),d\rangle_v=\sum_v \rho_{L,v}(h(d))=\rho_L(h(d))=0
$$
for some rational function $h$ on $A$. 
Now fix $\mathfrak{C}$ and consider the expression
$\sum_v\langle\mathfrak{C},d\rangle_v$.
We have just seen
that this depends
only on the linear equivalence class of $\mathfrak{C}$ (which
is translation invariant), and thus the translation invariance of each
$\langle\ ,\ \rangle_v$ shows that 
$\sum_v\langle\mathfrak{C},d\rangle_v$ is translation invariant
in the second variable (with $\mathfrak{C}$ held fixed).  
From this one may deduce
$$
\sum_v\langle\mathfrak{C},d\rangle_v
=\sum_v\langle\mathfrak{C},(b)-(0)\rangle_v,
$$
and so the left hand side depends only on $b$ and not on the choice of $d$.

Now suppose $X$ is a proper, smooth, geometrically connected
curve over $L$ with an $L$-rational point, and that $A$ is the
Jacobian of $X$. Let 
$\alpha:X\map{}A$ be the associated  canonical embedding.
For each place $v$ of $L$ we have from \S \ref{local symbols}
a $\Q_p$-valued symbol 
$\langle\ ,\ \rangle_{X_v,\rho_{L_v}}$
on disjoint divisors on $X_v=X\times_L L_v$.
By summing over all places, we obtain a symbol 
\begin{equation}\label{global curve}
\langle\ ,\ \rangle_{X,L}=\sum_v \langle\ ,\ \rangle_{X_v,\rho_{L_v}}
\end{equation}
defined on degree zero divisors of $X$ with disjoint support.  This pairing
descends to a (symmetric) pairing on linear equivalence classes
(this follows from Proposition \ref{curve height}(a,b) below
and the fact that $\rho=\sum_v\rho_{L_v}$ vanishes on $L^\times$).
In particular, $\langle\ ,\ \rangle_{X,L}$
extends bilinearly to all pairs of degree zero divisors, without the
assumption of disjoint support.

\begin{Rem}\label{Theta}
As $\langle\ ,\ \rangle_{X,L}$ is defined on linear equivalence classes,
it descends to a bilinear pairing
$$
\langle\ ,\ \rangle_{X,L}:A(L)\times A(L)\map{}\Q_p.
$$
which agrees with the pairing $-\langle\ ,\ \rangle_{A,L}$
when one identifies $A\iso A^\vee$ via the
canonical principal polarization \cite[\S 4.3]{pr1}.
\end{Rem}

\begin{Prop}\label{curve height}
Let $v$ be a prime of $L$ above a rational prime $\ell$.
The local N\'eron symbol 
$\langle C ,D \rangle_v=\langle C ,D \rangle_{X_v,\rho_{L_v}}$, 
defined on degree zero divisors on $X_v$ with disjoint support, satisfies
\begin{enumerate}
\item $\langle\ ,\ \rangle_v$ is symmetric and bilinear,

\item if $C=(f)$ is a principal divisor  then
$\langle C, D\rangle_v=\rho_{L_v}(f(D)),$

\item if $T$ is a correspondence from $X$ to itself and $T^\iota$
is the dual correspondence, then
$$
\langle TC,D\rangle_v=\langle C, T^\iota D\rangle_v,
$$

\item for $d_0\in X_v(L_v)-\mathrm{supp}(C)$, the function
on $X_v(L_v)-\mathrm{supp}(C)$
$$
d\mapsto\langle C,(d)-(d_0)\rangle_v
$$
is continuous for the $v$-adic topology,

\item if $\ell=p$, $L'$ is a finite extension of $L_v$ contained
in the cyclotomic $\Z_p$-extension of $L_v$, 
and $C$ and $D$ are  degree zero divisors on $X_v\times_{L_v}L'$ 
and $X_v$, respectively, then
$$
\langle \N_{L'/L_v} C, D\rangle_v
\subset c^{-1} \rho_{\Q_p}(\N_{L'/\Q_p}(L'))
$$
whenever this is defined, for some constant $c\in\Z_p$ independent of
$C$, $D$, and $L'$.
\end{enumerate}
Furthermore $\langle\ ,\ \rangle_v$ takes values in a compact subset of
$\Q_p$.
\end{Prop}

\begin{proof}
Properties (a)--(e) are direct consequences of the analogous properties of the
N\'eron symbol on $A$ in Proposition \ref{pr height}, except for the
symmetry (which is stated without proof in \cite{pr1}, but can be deduced
from the construction of the pairing of Proposition \ref{pr height}).
For the final claim one uses the finite generation of the $p$-primary 
part $A(L_v)$ as a $\Z_p$-module and the specified behavior on principal
divisors.
\end{proof}

\begin{Prop}\label{neron to intersection}
For any prime $v$ of $L$ with residue characteristic $\not=p$
and any degree zero divisors $C$ and $D$ on $X_v$ with disjoint support,
$$
\langle C, D\rangle_v = \log_p(\N(v))\ [C,D]
$$
where $[C,D]$ is the pairing of \S \ref{intersection theory}
for any regular, integral, proper scheme $\underline{X}$ 
over the integer ring of $L_v$ whose  generic fiber is $X_v$.
\end{Prop}

\begin{proof}
Using the discussion of \S \ref{intersection theory}, one can show
that the right hand side satisfies properties (a)--(d) of 
Proposition \ref{curve height}, and so it suffices to show
that these determine $\langle\ ,\ \rangle_v$ uniquely.  This is similar
to the uniqueness argument of Proposition \ref{pr height};
the difference of two such symbols would define a continuous bilinear
function $A(L_v)\times A(L_v)\map{}\Q_p$, which must be trivial for
topological reasons.
\end{proof}


\section{Intersections on modular curves}
\label{modular intersections}


Fix $s>0$ and $\sigma\in\Gal(H_s/K)$.
Let $\ell$ be a rational prime, $v$ a place of $H_s$ above $\ell$,
$F$ the completion of the maximal unramified extension of $H_{s,v}$,
$W$ the integer ring of $F$, and $\mathfrak{m}$ the maximal ideal of $W$. 
Set $W_n=W/\mathfrak{m}^{n+1}$. 
We denote by 
$X=X_0(N)_{/\Z}$ the canonical integral model of \cite{Katz-Mazur},
and set ${\underline X}=X\times_\Z W$.

\begin{Def}\label{hom notation}
Given  elliptic curves with $\Gamma_0(N)$-structure 
${\underline x}$ and ${\underline y}$
over $\Spec(W)$, we define  
$\Hom_{W_n}({\underline y},{\underline x})_{\deg(m)}$
to be the set of degree $m$ isogenies (of elliptic curves
with $\Gamma_0(N)$-structure, in the sense of \S \ref{CM}) 
$$
{\underline y}\times_W W_n\map{}
{\underline x}\times_W W_n.
$$
\end{Def}

\begin{Prop}\label{first homs}
Let ${\underline x}, {\underline y}\in {\underline X}(W)$
represent elliptic curves with $\Gamma_0(N)$-structure
over $W$, and assume that these sections 
intersect properly and reduce to regular, non-cuspidal points in the
special fiber.  Then
$$
i({\underline x},{\underline y})=\frac{1}{2}\sum_{n\ge 0}
|\Hom_{W_n}({\underline y},{\underline x})_{\deg(1)}|.
$$
\end{Prop}

\begin{proof}
This is \cite[Proposition III.6.1]{gz}, 
or \cite[Theorem 4.1]{conrad}.
\end{proof}

Now assume $\ell\not=p$ and
fix an integer $m=m_0p^r$ with $r>0$ and $(m_0,Np)=1$.
Choose an embedding $H_\infty\hookrightarrow F$ extending
$H_s\hookrightarrow F$.
Recall the notation 
$$
\bh_{s,r}=\Norm_{H_{s+r}/H_s}(h_{s+r})\hspace{1cm}
\bd_{s,r}=\Norm_{H_{s+r}/H_s}(d_{s+r})$$ 
of the introduction. For any $t\ge 0$, let $\uh_t$ 
be the Zariski closure (with the reduced subscheme structure) of 
$h_t\in X(F)$ in ${\underline X}$
and let $T_{m_0}(\ubh^\sigma_{s,r})$ be the horizontal Weil divisor
on ${\underline X}$ with generic fiber $T_{m_0}(\bh^\sigma_{s,r})$.
By the valuative criterion of properness, the closed subscheme $\uh_{s+r}$ 
has the form $\Spec(W)\map{}{\underline X}$.
Moreover, the section $\uh_{s+r}$ 
arises from a Heegner diagram defined over $W$.
Indeed, by \cite[Proposition 1.2]{cornut} or \cite[Theorems 8,9]{st}
the point $h_{s+r}\in X(H_{s+r})$
arises from a Heegner diagram over $H_{s+r}$ with good reduction above
$\ell$, and so the section $\uh_{s+r}$ represents the 
N\'eron model over $W$ of this
Heegner diagram.  Taking the quotient
of $\uh_{s+r}$ by its $p\co_{s+r-1}$-torsion, we obtain a 
Heegner diagram represented by the section 
$\uh_{s+r-1}\in\underline{X}(W)$, and so on through all lower conductors.
In particular we now have a $p$-isogeny of Heegner diagrams
defined over $W$
\begin{eqnarray}\label{Heegner square}\xymatrix{
{\uE_s\ar[r]^{\uh_s}}\ar[d]_\phi&\uE_s'\ar[d]_{\phi'}\\
\uE_{s-1}\ar[r]^{\uh_{s-1}} &\uE_{s-1}'.
}\end{eqnarray}

Although the expression for the local N\'eron symbol at $\ell\not=p$
in terms of intersection theory requires working on a regular model
(which ${\underline X}$ is not when $\ell|N$)
and modifying the divisors in questions by a fibral divisor, in our
situation these details can be ignored:

\begin{Prop}\label{height to intersections}
Suppose $\ell\not=p$ and $0\le t\le s$. Then
$$
\langle c_t, T_{m_0}(\bd_{s,r}^\sigma)\rangle_v
=
\log_p(\N(v))\cdot i(\uh_t, T_{m_0}(\ubh^\sigma_{s,r})),
$$
where the pairing on the left is the local N\'eron symbol on $X_{/H_{s,v}}$
of Proposition \ref{curve height}
and $i$ is the intersection multiplicity on ${\underline X}$ 
of \S \ref{intersection theory}.
\end{Prop}

\begin{proof}
As in \cite[Proposition III.3.3]{gz}, together with Proposition 
\ref{neron to intersection}.
\end{proof}

\begin{Rem}
In order to make sense of 
$i(\uh_t, T_{m_0}(\ubh^\sigma_{s,r}))$ when
$\ell|N$ we need to justify why the prime Weil divisors occuring
in $T_{m_0}(\ubh^\sigma_{s,r})$ are  locally principal, 
so that $T_{m_0}(\ubh^\sigma_{s,r})$ may be viewed
as a Cartier divisor.  The geometric points of $T_{m_0}(\bh^\sigma_{s,r})$
all occur in the support of $T_m (h_s^\sigma)$.  If $\ell|N$
then these points represent Heegner diagrams 
which are prime-to-$\ell$ isogenous to $h_s^\sigma$, and so are all
defined over $F$.  Arguing as in \cite[Corollary 2.7]{conrad} (Conrad's
$p$ is our $\ell$),
the Zariski closures of these points on ${\underline X}$
are sections to the structure map ${\underline X}\map{} \Spec(W)$
and lie in the smooth locus.  In particular,
the associated ideal sheaves are locally free of rank one.
\end{Rem}

\begin{Prop}\label{split primes}
Suppose $\ell\not=p$ and $\epsilon(\ell)=1$.  Then  for all $0\le t\le s$,
$\langle c_t, T_{m_0}(\bd_{s,r}^\sigma)\rangle_v=0$,
where the pairing $\langle\ ,\ \rangle_v$ is as in Proposition 
\ref{height to intersections}.
\end{Prop}

\begin{proof}
By Proposition \ref{height to intersections} we must show that
$i(\uh_t, T_{m_0}(\underline{\bh}_{s,r}^\sigma))=0$.
The claim is unchanged if we replace $W$ by the integer ring of
a finite extension of $F$.
Doing so, we assume that the divisor $T_{m_0}({\bh}_{s,r}^\sigma)$
is defined point-by-point over $F$ and that
the horizontal divisor $T_{m_0}(\bd_{s,r}^\sigma)$ on 
$\underline{X}$ is a sum of sections to the structure map,
 each of which 
represents a Heegner diagram over $W$ whose conductor divides $mp^s$
and has exact valuation $s+r>t$ at $p$.  
Let $\underline{x}$ be one such Heegner diagram, and let $\co$ and $\co'$
be the endomorphism rings of  $\underline{x}$ and its closed fiber,
respectively. These are orders in $K$, 
as $\underline{x}$ has ordinary reduction, and $\co\subset\co'$.
By the Serre-Tate theorem, $\co$ is the intersection (in
$K\otimes \Q_\ell$) of  $\co'$ and $\co\otimes\Z_\ell$,
therefore 
$$
\ord_p(\mathrm{cond}(\co'))=\ord_p(\mathrm{cond}(\co))
=s+r>t.
$$
The same argument shows that the valuation at $p$ of the conductor of 
the CM order of the special fiber of $\uh_t$ is $t$, and so
the Heegner diagram $\uh_t$ is distinct in the special fiber from 
all Heegner diagrams appearing in  
$T_{m_0}({\underline \bh}_{s,r}^\sigma)$.  By Proposition \ref{first homs},
$i(\uh_t, T_{m_0}(\underline{\bh}_{s,r}^\sigma))=0$.
\end{proof}


\section{Nonsplit primes away from $p$}
\label{nonsplit}


In this section we examine the local N\'eron pairings between Heegner 
points at places lying above rational primes $\not=p$ which are
nonsplit in $K$.  The methods are based on those of 
Chapter III of \cite{gz}, and this portion of Gross and Zagier's
work has been reworked and rewritten by Conrad \cite{conrad}
with the addition of considerably more detail.

Keep the notation of \S \ref{modular intersections}, and assume
$\ell\not=p$ is nonsplit in $K$.  In particular $\ell\nmid N$.
Fix a prime $v$ of $H_s$ (with $s>0$, as always) above $\ell$ and 
an integral $\co_s$-ideal $\fa$ of norm prime to $D\ell p$ whose class in
$\Pic(\co_s)$ represents $\sigma$ under the Artin map.
We denote by $\fl$ the unique prime of $\co_s$ above $\ell$
(we sometimes let $\fl$ denote the $\co_K$-ideal $\fl\co_K$; a mild abuse
of notation).
If $\epsilon(\ell)=-1$ then $\fl=\ell\co_s$ is trivial in $\Pic(\co_s)$,
$\fl$ splits completely in $H_s$,
and $v$ has absolute residue degree $2$.
If $\epsilon(\ell)=0$
then $\fl^2=\ell\co_s$ and $\fl$ is not a principal ideal of 
$\co_s$ (if $D$ is not prime then $\fl\co_K$ is not principal,
if $D=-\ell$ is prime then $\fl=(\sqrt{D}\cap\co_s)$ is not
principal since $s>0$).  Thus when $\epsilon(\ell)=0$,
$\fl$ has order $2$ in $\Pic(\co_s)$ and again $v$ has
residue degree $2$.


\subsection{Intersection via $\Hom$ sets}



\begin{Prop}\label{intersections}
For any integer $m=m_0p^r$ with $(m_0,Np)=1$,
\begin{align*}
\langle c_s, T_{m_0}(\bd_{s,r+2}^\sigma)\rangle_v &=
\log_p(\ell)\sum_{n\ge0}
\Big(
|\Hom_{W_n}( {\uh}_s^\fa, {\uh}_s)_{\deg(mp^2)}|-
|\Hom_{W_n}( {\uh}_{s-1}^\fa, {\uh}_s)_{\deg(mp)}|
\Big) \\
\langle c_{s-1}, T_{m_0}(\bd_{s,r+1}^\sigma)\rangle_v &=
\log_p(\ell)\sum_{n\ge0}\Big(
|\Hom_{W_n}( {\uh}_{s}^\fa, {\uh}_{s-1})_{\deg(mp)}|
-|\Hom_{W_n}( {\uh}_{s-1}^\fa, {\uh}_{s-1})_{\deg(m)}|\Big)
\end{align*}
where $\langle\ ,\ \rangle_v$ is the local N\'eron symbol on $X_{/H_{s,v}}$
of Proposition \ref{curve height}, and the $\Hom$ sets are those of 
Definition \ref{hom notation}.
\end{Prop}
\begin{proof}
We will prove the first equality.  The proof of the second involves
only a change of subscripts.

First consider the easy case where $(\ell,m_0)=1$.  Then the divisor
$T_{m_0}(\bh^\sigma_{s,r+2})$ on $\underline{X}_{/F}$, (recall that $F$ is
the completion of the maximal unramified extension of $H_{s,v}$, $W$
is its integer ring, and $\underline{X}=X_0(N)_{/W}$) 
is a sum of sections to the structure map.  
Hence the same is
true of the horizontal divisor $T_{m_0}(\ubh_{s,r+2}^\sigma)$
on $\underline{X}$, and each section represents a Heegner diagram over
$\Spec(W)$.  Namely, if we fix an extension of $\sigma$ to 
$\Gal(H_{s+r+2}/K)$ and an ideal $\fa$ of $\co_{s+r+2}$ representing
this extension, then 
\begin{equation}\label{easy hom}
T_{m_0}(\ubh^\sigma_{s,r+2})=
\sum_{\fb}\sum_C \uh^{\fa\fb}_{s+r+2/C}
\end{equation}
where $\fb$ runs over classes in $\Pic(\co_{s+r+2})$ which are trivial
in $\Pic(\co_s)$, $C$ runs over the order $m_0$-subgroup schemes
of the Heegner diagram $\uh^{\fa\fb}_{s+r+2}$ over $\Spec(W)$
and the subscript $/C$ means the quotient by $C$
(which makes sense since $(m_0,N)=1$). Since $\ell$ does not divide
$m_0$, each $C$ is \'etale (in fact constant), 
determined uniquely by its reduction 
to $W_n$ for any $n$, and the decomposition 
(\ref{easy hom}) holds over $W_n$.  By Proposition \ref{first homs}
\begin{eqnarray*}
i(\uh_s, T_{m_0}(\ubh^\sigma_{s,r+2}))&=& 
\sum_{\fb}\sum_C i(\uh_s,\uh^{\fa\fb}_{s+r+2/C})\\
&=&\frac{1}{2}\sum_n\sum_{\fb}\sum_C 
|\Hom_{W_n}(\uh^{\fa\fb}_{s+r+2/C}, \uh_s)_{\deg(1)}|\\
&=&\frac{1}{2}\sum_n\sum_{\fb}
|\Hom_{W_n}(\uh^{\fa\fb}_{s+r+2}, \uh_s)_{\deg(m_0)}|,
\end{eqnarray*}
and by Proposition \ref{height to intersections} the first
equality of Proposition \ref{intersections} follows once we show 
\begin{equation}
\label{hom split easy}
|\Hom_{W_n}( {\uh}_{s}^\fa, {\uh}_{s})_{\deg(mp^2)}|
= |\Hom_{W_n}( {\uh}_{s-1}^\fa, {\uh}_{s})_{\deg(mp)}|
+\sum_{\fb}|\Hom_{W_n}(\uh^{\fa\fb}_{s+r+2}, \uh_s)_{\deg(m_0)}|.
\end{equation}

The $p^{r+2}$-torsion on $\uh_s^\fa$ is constant as a group scheme,
and so the kernel of any degree $mp^2$ isogeny
$f:\uh^\fa_s\map{}\uh_s$ over $W_n$
determines an order $p^{r+2}$-subgroup
of $\uh_s^\fa(W)$.    By the Euler system relations of 
\S \ref{Hecke action},
every such subgroup is either the kernel of a map which factors through
$\phi^\fa:\uh_s^\fa\map{}\uh_{s-1}^\fa$, 
or is the kernel of the dual isogeny to
$\phi^{\fa\fb}\circ\cdots\circ\phi^{\fa\fb}:
\uh_{s+r+2}^{\fa\fb}\map{}\uh_s^\fa$ for some choice of $\fb$,
and the two cases are mutually exclusive.
Thus $f$ has one of the two forms
$$
\uh_s^\fa\map{\phi^\fa}\uh_{s-1}^\fa\map{\psi}\uh_s
\hspace{2cm}
\uh_s^\fa\map{(\phi^{\fa\fb}\circ\cdots\circ \phi^{\fa\fb})^\vee }
\uh_{s+r+2}^{\fa\fb}\map{\psi}\uh_s
$$
where  $\psi$ has degree either $mp$ or $m_0$ (respectively).
The equality (\ref{hom split easy}) follows.

Now consider the case where $\ell$ divides $m_0$.  This is considerably more
involved, but nearly all of what we need is covered by the generality of 
\cite[\S 6]{conrad} (which is based on \cite[III \S 4--6]{gz}), to which 
we refer the reader for the proof of (\ref{hard homs}) below.
Write $m_0=m_1\ell^t$ with $(\ell,m_1)=1$.
As above, the divisor $T_{m_1}(\ubh^\sigma_{s,r+2})$ on $\underline{X}$
is a sum of sections, each of which represents a Heegner diagram over
$\Spec(W)$, and we denote by $Z$ the set of such sections
$$Z=\{\uh^{\fa\fb}_{s+r+2/C}\mid \fb\in\mathrm{Ker}(\Pic(\co_{s+r+2})
\map{}\Pic(\co_s)) \}
$$
where $C$ runs over the order $m_1$ subgroup schemes of 
$\uh^{\fa\fb}_{s+r+2}$.  For each $\uz\in Z$, one has the expected 
(but much more subtle) equality
\begin{eqnarray}\nonumber
i(\uh_s,T_{m_0}(\uh^\fa_{s+r+2}))&=&
\sum_{\uz\in Z}i(\uh_s,T_{\ell^t}(\uz))\\\label{hard homs}
&=&\frac{1}{2}\sum_{\uz\in Z}\sum_{n\ge 0}
|\Hom_{W_n}(\uz, \uh_s)_{\deg(\ell^t)}|\\
\nonumber &=&\frac{1}{2}\sum_n\sum_{\fb}
|\Hom_{W_n}(\uh^{\fa\fb}_{s+r+2}, \uh_s)_{\deg(m_0)}|.
\end{eqnarray}
With this in hand, the remainder of the proof is exactly as in the
case $(\ell,m_0)=1$.
\end{proof}


\subsection{Inclusion-Exclusion}


Our goal is, for any positive 
integer $m$ with $(m,N)=1$, to express the sum over $n$ of
\begin{eqnarray}\label{inclusion}\lefteqn{
|\Hom_{W_n}( {\uh}_s^\fa, {\uh}_s)_{\deg(mp^2)}|
-|\Hom_{W_n}( {\uh}_{s-1}^\fa, {\uh}_s)_{\deg(mp)}| } \hspace{1cm}\\  \nonumber
& &
-|\Hom_{W_n}( {\uh}_{s}^\fa, {\uh}_{s-1})_{\deg(mp)}|
+|\Hom_{W_n}( {\uh}_{s-1}^\fa, {\uh}_{s-1})_{\deg(m)}|
\end{eqnarray}
as a sum over elements in the quaternion algebra
$B=\End_{W_0}(\uh_s)\otimes_{\Z}\Q$.

\begin{Lem}
Base change to the   fiber induces a degree preserving 
injection 
$$\Hom_{W_n}( {\uh}_s^\fa, {\uh}_s)\map{}\Hom_{W_0}( {\uh}_s^\fa, {\uh}_s),$$
and similarly for the other $\Hom$ sets occuring in (\ref{inclusion}).
\end{Lem}
\begin{proof}
This is \cite[Lemma 2.1(2)]{conrad} or \cite[Proposition VI.2.4(2)]{Goren}.
\end{proof}

The isogeny $\phi$ induces injections
$$
\Hom_{W_n}( {\uh}_{s-1}^\fa, {\uh}_s)
\map{\circ\phi^\fa}\Hom_{W_n}(\uh_s^\fa,\uh_s)
\map{}\Hom_{W_0}(\uh_s^\fa,\uh_s)
$$
$$
\Hom_{W_n}( {\uh}_{s}^\fa, {\uh}_{s-1})
\map{\phi^\vee\circ}\Hom_{W_n}(\uh_s^\fa,\uh_s)
\map{}\Hom_{W_0}(\uh_s^\fa,\uh_s)
$$
whose images we denote by $L_n$ and $L^\vee_n$, respectively.
We also define $M_n$ to be the image of the injective composition
$$\Hom_{W_n}( {\uh}_{s-1}^\fa, {\uh}_{s-1})\map{}
\Hom_{W_n}(\uh_s^\fa,\uh_s)\map{}
\Hom_{W_0}(\uh_s^\fa,\uh_s)
$$
where the first arrow is given by
$f\mapsto \phi^\vee\circ f\circ \phi^\fa$. 
Clearly $M_n\subset L_n\cap L^\vee_n$.  The scheme-theoretic kernels
$$\ker\big(\phi: \uE_s\map{}\uE_{s-1}\big)
\hspace{1cm}
\ker\big(\phi^\fa: \uE_s^\fa\map{}\uE_{s-1}^\fa\big)
$$
are constant group schemes of order $p$ over $W$. We define
$$C=(\ker\ \phi)(W_0)\hspace{1cm}
C^\fa=(\ker\ \phi^\fa)(W_0).$$

\begin{Def}
We will say that $f\in \Hom_{W_n}(\uh_s^\fa,\uh_s)$ is 
\emph{stable} if the restriction of $f$ to the fiber
$f_0:\uE_s^\fa(W_0)\map{}\uE_s(W_0)$ takes $C^\fa$ into $C$.
We will say that $f$ is \emph{unstable} otherwise, and make similar
definitions for maps from $\uh_s$ to itself.  
If $Z\subset \Hom_{W_n}(\uh_s^\fa,\uh_s)$ is any subset, we will
write $Z^{\mathrm{stable}}$ and $Z^{\mathrm{unstable}}$
for the subsets of stable and unstable elements of $Z$.
\end{Def}

\begin{Lem}\label{IEi}
Suppose $m$ is any positive integer with  $(m,N)=1$.
Base change to the   fiber identifies the stable
elements of degree $mp^2$ in $\Hom_{W_n}(\uh_s^\fa,\uh_s)$  
with the degree $mp^2$ elements of $L_n\cup L^\vee_n$.
\end{Lem}

\begin{proof}
Fix $f\in\Hom_{W_n}(\uh_s^\fa,\uh_s)$ of degree divisible by $p$ 
and prime to $N$. Letting 
$f_0$ denote the restriction of $f$ to
geometric points as above, $f$ is stable if and only if
either $f_0(C^\fa)=0$ or $f_0(C^\fa)=C$.  
The first condition is equivalent to $f_0=g_0\circ\phi^\fa$
for some $g_0\in\Hom_{W_0}(\uE_{s-1}^\fa,\uE_s)$.  Since $\phi^\fa$
has degree $p$ it induces an isomorphism on $\ell$-divisible groups
over $W_n$, and so the map on $\ell$-divisible groups induced by $g_0$
lifts to $W_n$.  By the Serre-Tate theorem $g_0$ itself lifts
to a morphism over $W_n$, and so $f\in L_n$.  
Now suppose $f_0(C^\fa)=C$. Since the degree of $f$ is divisible by $p$
we must have $f_0(\uE_s^\fa(W_0)[p])=C$, and so 
$f^\vee_0(C)=(f^\vee_0\circ f_0)(\uE_s^\fa(W_0)[p])=0$.  
Hence $f^\vee_0=g_0\circ\phi$ for some 
$g_0\in\Hom_{W_0}(\uE_{s-1},\uE_s^\fa)$, and so $f_0\in L^\vee_n$ as above.

Conversely, if $f_0\in L_n\cup L^\vee_n$
then either $f_0(C^\fa)=0$ or $f^\vee_0(C)=0$.  In the second
case we compute the Weil $e_p$-pairing 
$$
e_p\big( f_0(\uE_s^\fa(W_0)[p]),C \big) =
e_p\big( \uE_s^\fa(W_0)[p], f^\vee_0(C)\big) =0.
$$
This implies $f_0(\uE_s^\fa(W_0)[p])\subset C$, and so, in either case,
$f_0(C^\fa)\subset C$ and $f$ is stable.
\end{proof}

\begin{Lem}\label{IEii}
For any positive integer $m$ with $(m,N)=1$, 
the composition 
$$
\Hom_{W_n}(\uh_s^\fa,\uh_s)\map{}\Hom_{W_0}(\uh_s^\fa,\uh_s)\map{p}
\Hom_{W_0}(\uh_s^\fa,\uh_s)
$$ taking $f\mapsto pf_0$
identifies the unstable elements of 
$\Hom_{W_n}(\uh_s^\fa,\uh_s)_{\deg(m)}$ with 
the complement of $(M_n)_{\deg(mp^2)}$ in 
$(L_n\cap L^\vee_n)_{\deg(mp^2)}$ (the degree $mp^2$ elements of
$M_n$ and $L_n\cap L^\vee_n$, respectively).
\end{Lem}

\begin{proof}
First suppose we are given some $f\in\Hom_{W_n}(\uh_s^\fa,\uh_s)$;
the claim is that $pf_0\in M_n$ if and only if $f$ is stable.  
By definition $pf_0\in M_n$ if and only if there is some
$f'\in\Hom_{W_n}(\uE_{s-1}^\fa,\uE_{s-1})$ such that 
$pf=\phi^\vee\circ f'\circ \phi^\fa$, or equivalently, such that 
$\phi\circ f= f'\circ\phi^\fa$.  Furthermore, this is equivalent
to finding $f'_0\in \Hom_{W_0}(\uE_{s-1}^\fa,\uE_{s-1})$ such that
$\phi\circ f_0= f'_0\circ\phi^\fa$ holds in the   fiber
(since $\phi$ and $\phi^\fa$ induce isomorphisms on 
$\ell$-divisible groups over $W_n$, the map on $\ell$-divisible
groups induced by $f_0'$ lifts to $W_n$, and so the Serre-Tate
theorem implies that $f_0'$ itself lifts).  Such an $f'_0$ exists
if and only if $(\phi\circ f_0)(C^\fa)=0$, which is equivalent to
$f$ being stable.

Now suppose we are given a homomorphism 
$g_0\in L_n\cap L^\vee_n$ of degree
divisible by $p^2$, with $g_0\not\in M_n$.  
There is some $y\in\Hom_{W_n}(\uh_s^\fa,\uh_{s-1})$ such that
$g_0$ is the restriction of $g=\phi^\vee\circ y$ to the   fiber.
Let $y_0$ denotes the restriction of $y$ to the   fiber.
If $y_0(C^\fa)=0$ we could write $y_0=y_0'\circ\phi^\fa$ for some 
$y_0'\in\Hom_{W_0}(\uE_{s-1}^\fa,\uE_{s-1})$.
As above, the map on $\ell$-divisible groups induced by such a $y_0'$ would
lift to $W_n$, and so by the Serre-Tate theorem 
$y_0'$ itself would lift to some
$y'\in\Hom_{W_n}(\uE_{s-1}^\fa,\uE_{s-1})$ with $g_0$ equal to the
restriction of $\phi^\vee\circ y'\circ\phi^\fa$ to the   
fiber. This contradicts $g_0\not\in M_n$,  so $y_0(C^\fa)\not=0$.
Since $p$ divides the degree of $y_0$ we must have 
$y_0(\uE_s^\fa(W_0)[p])=y_0(C^\fa)$.  Now 
$g_0\in L_n$ implies
$$
0=g_0(C^\fa)=(\phi^\vee_0\circ y_0)(C^\fa)=g_0(\uE_s^\fa(W_0)[p]),
$$
so $g_0=p f_0$ for some $f_0\in\Hom_{W_0}(\uE_s^\fa,\uE_s)$.
As above, the Serre-Tate theorem guarantees that $f_0$ lifts to
a morphism $f$ over $W_n$.
\end{proof}

\begin{Cor}\label{IE}
The expression (\ref{inclusion}) is equal to
\begin{eqnarray*}
|\Hom_{W_n}( \uh_s^\fa, \uh_s)_{\deg(mp^2)}^\unstable|
-|\Hom_{W_n}( \uh_s^\fa, \uh_s)_{\deg(m)}^\unstable|.
\end{eqnarray*}
\end{Cor}

\begin{proof}
By the definitions of $M_n$, $L_n$ and $L^\vee_n$,
\begin{eqnarray*}
|\Hom_{W_n}( \uh_{s-1}^\fa, \uh_{s-1})_{\deg(m)}| &=& |(M_n)_{\deg(mp^2)}|\\
|\Hom_{W_n}( \uh_{s-1}^\fa, \uh_{s})_{\deg(mp)}| &=& |(L_n)_{\deg(mp^2)}|\\
|\Hom_{W_n}( \uh_{s}^\fa, \uh_{s-1})_{\deg(mp)}| &=& 
|(L^\vee_n)_{\deg(mp^2)}|
\end{eqnarray*}
Consequently, the expression (\ref{inclusion}) is equal to
\begin{eqnarray*}\lefteqn{
|\Hom_{W_n}( \uh_s^\fa, \uh_s)_{\deg(mp^2)}^\unstable| +
  |\Hom_{W_n}( \uh_s^\fa, \uh_s)_{\deg(mp^2)}^\stable| } \hspace{3cm}\\
& &   -|(L_n)_{\deg(mp^2)}|-|(L^\vee_n)_{\deg(mp^2)}|+
|(M_n)_{\deg(mp^2)}|.
\end{eqnarray*}
By Lemma \ref{IEi} this is
\begin{eqnarray*}
\lefteqn{|\Hom_{W_n}( \uh_s^\fa, \uh_s)_{\deg(mp^2)}^\unstable| +
  |(L_n\cup L^\vee_n)_{\deg(mp^2)}| } \hspace{3cm} \\
& &-|(L_n)_{\deg(mp^2)}|-|(L^\vee_n)_{\deg(mp^2)}|+
|(M_n)_{\deg(mp^2)}|
\end{eqnarray*}
which we write as
\begin{eqnarray*}
\lefteqn{
|\Hom_{W_n}( \uh_s^\fa, \uh_s)_{\deg(mp^2)}^\unstable| -
  |(L_n\cap L^\vee_n)_{\deg(mp^2)}| + |(M_n)_{\deg(mp^2)}| } \hspace{2cm} \\
&=&|\Hom_{W_n}( \uh_s^\fa, \uh_s)_{\deg(mp^2)}^\unstable| -
|\Hom_{W_n}( \uh_s^\fa, \uh_s)_{\deg(m)}^\unstable|
\end{eqnarray*}
using Lemma \ref{IEii}.
\end{proof}

Set $R=\Hom_{W_0}(\uh_s,\uh_s)$ and $B=R\otimes_{\Z}\Q$.
Thus $B$ is a rational quaternion 
algebra ramified exactly at $\ell$ and $\infty$,
and  $R\subset B$ is a level-$N$ Eichler 
order \cite[Lemma 7.1]{conrad}. 
The reduction map $$\Hom_{W}(\uh_s,\uh_s)\map{}\Hom_{W_0}(\uh_s,\uh_s)$$
induces an embedding $\iota:K\map{}B$ which, by the Serre-Tate theorem, 
is \emph{optimal} for the pair $(\co_s,R)$ in the sense that 
$\iota(K)\cap R=\iota(\co_s)$.
We henceforth regard $K$ as a subfield of $B$, supressing $\iota$
from the notation.  There is a canonical decomposition
$$B= B^+\oplus B^-=K\oplus Kj$$ where $j\in B$ is 
a trace zero element with the property $j x j^{-1}=\bar{x}$
for all $x\in K$.  This characterizes $j$ up to multiplication
by $\Q^\times$. The reduced norm is additive with respect to this
decomposition, i.e. $\N(b^++b^-)=\N(b^+)+\N(b^-)$.
We wish to determine which $b\in R=\Hom_{W_0}(\uh_s,\uh_s)$
are unstable.

\begin{Lem}\label{unstable}
An endomorphism $b\in R$ is unstable if 
and only if $$\ord_p\N(b^+)=\ord_p\N(b^-)=-2s,$$
where $b^\pm$ is the projection of $b$ to the summand $B^\pm$.
\end{Lem}

\begin{proof}
We are free to assume that $j$ is chosen in $R$.
Let $T$ denote the $p$-adic Tate module of
$\uE_s(W_0)[p^\infty]$ and set $V=T\otimes\Q_p$.
The split quaternion algebra $B_p=B\otimes\Q_p$ acts on 
$V$, and the stabilizer of
$T\subset V$ is exactly $R_p=R\otimes\Z_p$ (since the order $R$ is locally
maximal away from $N$).
Under the identification of $V/T$ with $\uE_s(W_0)[p^\infty]$, the subgroup
$\co_{s-1,p}T/T$ is identified with $C$, and so the unstable elements
of $R$ are exactly those which do not stabilize the lattice
$T'=\co_{s-1,p}T\supset T$.  As an $\co_{s,p}$-module, $T$ is free of
rank one (proof: $T$ is isomorphic as an $\co_{s,p}$-module to some
fractional $\co_{s,p}$-ideal.  By the optimality of 
$K\map{}B$ with respect to $(\co_s,R)$, this ideal is proper, and  
all proper ideals of $\co_{s,p}$ are principal).
Fix a generator $t\in T$, and let
$X\in \co_{s,p}$ be such that $jt=Xt$.  This implies 
in particular that $\N(X)=\N(j)$.
As a $\Z_p$-module, $T$ is generated by $t$ and $p^s\sqrt{D}t$, and
so $\alpha+\beta j\in B$ (with $\alpha$, $\beta\in K$)
stabilizes $T$ if and only if
the elements 
$$(\alpha+\beta j)t=(\alpha+\beta X)t\hspace{1cm}
(\alpha+\beta j)p^s\sqrt{D}t=(\alpha-\beta X)p^s\sqrt{D}t$$
are in $T$.  From this we deduce that
$$R_p=\{\alpha+\beta j\in B_p \mid 
\alpha,\beta X \in (p^s\sqrt{D})^{-1}\co_{s,p},
\alpha+\beta X\in\co_{s,p} \}.$$
Applying similar reasoning to the lattice $T'$, we find that
the order of $B_p$ leaving both $T$ and $T'$ stable is
$$ R_p^\mathrm{stable}=\{\alpha+\beta j\in B_p \mid 
\alpha,\beta X\in (p^{s-1}\sqrt{D})^{-1}\co_{s-1,p},
\alpha+\beta X\in\co_{s,p} \}.$$

Given  $b=\alpha+\beta j\in R_p$, set $\alpha'=p^s\sqrt{D}\alpha$ and
$\beta'=p^s\sqrt{D}\beta$. 
It is easily seen that the set of elements of 
$\co_{s,p}$ of norm divisible by $p$ is
equal to the unique maximal ideal $p\co_{s-1,p}\subset \co_{s,p}$.
Since $\alpha'\equiv -\beta' X\pmod{p^s\sqrt{D}\co_{s,p}}$,
$\alpha'$ is a unit if and only if $\beta' X$ is a unit.
Both elements are units if and only if 
$\ord_p\N(\alpha)=\ord_p\N(\beta X)=-2s$,
and both are nonunits if and only if $\alpha+\beta j\in R_p^\mathrm{stable}$.
\end{proof}

\begin{Prop}
For any non-negative integers $m$, $n$ with $(m,N)=1$, 
there is a bijection between 
$\Hom_{W_n}( \uh_s^\fa, \uh_s)_{\deg(m)}^\unstable$
and the set of all $b\in R\cdot\fa$ such that
\begin{enumerate}
\item $\N(b)=m\N(\fa)$,
\item $\ord_p\N(b^+)=\ord_p\N(b^-)=-2s$,
\item and 
$$
\ord_\ell \big(D\N(b^-)\big) \ge 
\left\{\begin{array}{ll} 2n+1& \mathrm{if\ } \epsilon(\ell)=-1\\
 n+1 & \mathrm{if\ } \epsilon(\ell)=0.\end{array}\right.
$$
\end{enumerate}
\end{Prop}

\begin{proof}
By \cite[Proposition III.7.3]{gz} or \cite[Theorem 7.12 and (7-3)]{conrad}
there is an isomorphism of left $\co_s$-modules
$$
\Hom_{W_n}( \uh_s^\fa, \uh_s)\iso \Hom_{W_n}(\uh_s,\uh_s)\otimes_{\co_s}\fa
$$
whose image (viewed as a lattice in $R\fa$) 
is exactly those elements satisfying property (c),
under which the degree $m$ isogenies correspond to those 
satisfying property (a).
We must show that this bijection takes the stable elements
onto those $b=b^++b^-$ for which property (b) fails.  
The isomorphism in question is defined as follows.  The map
$$\End_{W_n}(\uE_s)\otimes_{\co_s}\fa\map{\xi_n}
\Hom_{W_n}(\Hom_{\co_s}(\fa,\uE_s),\uE_s)\iso\Hom_{W_n}(\uE_s^\fa,\uE_s)$$
defined by $\xi_n(f\otimes x)(\phi)=f(\phi(x))$
is an isomorphism of $\co_s$-modules by Lemma 7.13 of \cite{conrad}, 
and taking level $N$ stucture into account we obtain an injection
of left $\co_s$-modules
$$
\Hom_{W_n}(\uh_s^\fa,\uh_s)\iso\Hom_{W_n}(\uh_s,\uh_s)\otimes_{\co_s}\fa
\hookrightarrow R\fa.
$$
This injection identifies 
$$\Hom_{W_n}(\uh_s^\fa,\uh_s)^\stable\iso
\Hom_{W_n}(\uh_s,\uh_s)^\stable\otimes_{\co_s}\fa$$
inside of $R\fa$
(this is easily checked everywhere locally using the fact that $\fa$ is 
proper, hence locally principal).  Localizing at $p$ and using 
$(\N(\fa),p)=1$, the claim follows from  Lemma \ref{unstable}.
\end{proof}

For any order $S$ of $B$, define
\begin{eqnarray}
\nonumber D_s^\fa(S,m)&=&\left\{ b\in S\cdot\fa \left|
\begin{array}{c} \N(b)=m\N(\fa)\\ 
\ord_p\N(b^+)=\ord_p\N(b^-)=-2s
\end{array}\right.\right\}\\
\label{delta def}\Delta_s^\fa(S,m) &=&  \sum_{b\in D_s^\fa(S,m)}
\left\{ \begin{array}{rr}
\frac{1}{2}\big(1+\ord_\ell\N(b^-)\big)
& \mathrm{if\ }\epsilon(\ell)=-1\\
\ord_\ell\big( D\N(b^-)\big) & \mathrm{if\ }\epsilon(\ell)=0
\end{array}\right.
\end{eqnarray}

\begin{Cor}\label{unstable to quaternion}
For $(m,N)=1$, $$\sum_{n\ge 0}
|\Hom_{W_n}( \uh_s^\fa, \uh_s)_{\deg(m)}^\unstable|
=\Delta_s^\fa(R,m).$$
\end{Cor}
\begin{proof}
When $\epsilon(\ell)=0$ this is immediate from the proposition above.
When $\epsilon(\ell)=-1$ it is similarly clear, provided one
knows that $\ord_\ell\N(b^-)$ is always odd; but (as we will see in 
the next section) we are free to choose $j$ in such a way that
$\ord_\ell\N(j)=1$, so writing  $b^-=\beta j$ with $\beta\in K$, 
$\ord_\ell(\N(b^-))=1+\ord_\ell\N(\beta)$ is odd.
\end{proof}


\subsection{Quaternionic sums}
\label{quaternions}


We continue to let $B$ be the rational quaternion algebra of discriminant
$\ell$ and assume we have a fixed embedding $K\hookrightarrow B$.
As noted before, this embedding induces a splitting 
$B=B^++B^-=K\oplus Kj$.
Let $\mathcal{S}$ denote the (finite) set of $K^\times$-conjugacy classes 
of $\co_s$-optimal, level $N$ Eichler orders in $B$;  
that is, level $N$ Eichler orders
$S$ such that $S\cap K=\co_s$, modulo the conjugation action of
$K^\times$. For such an $S$, the value of
$\Delta_s^\fa(S,m)$ (defined in (\ref{delta def}))
depends only on the class of $S$ in $\mathcal{S}$. 
Define 
$$
\Delta^\fa_s(m)=
\sum_{S\in\mathcal{S}}\Delta^\fa_s(S,m).
$$

The remainder of this section is devoted to the proof of the
following proposition.  The statement holds without parity restrictions
on $D$, but \emph{we will assume throughout that $D$ is odd}, refering the
reader to \cite{mann} for a description of the needed changes to the proof
in the case where $D$ is even.  
The method of proof  follows the 
calculations performed in \cite[\S III.9]{gz} (and described in great detail
in \cite{mann}).  The main difference (apart from working 
in higher conductor) is that
we have ``removed the Euler factor at $p$'' by adding the 
condition $\ord_p\N(b^+)=\ord_p\N(b^-)=-2s$ to the set $D_s^\fa(S,m)$
over which the summation $\Delta^\fa_s(S,m)$ occurs.

\begin{Prop}\label{quaternion sum}
There is a proper integral $\co_s$-ideal $\fq$ such that
for every positive integer $m$
\begin{equation*}
\Delta_s^\fa(m)=
\sum_{ \genfrac{}{}{0pt}{}{n>0}{\ell|n,(n,p)=1} } 
\delta(n) r_\fa(mp^{2s}|D|-nN)\cdot
\left\{\begin{array}{ll}
\ord_\ell(\ell n) R_{\fa\fq\fn}(n/\ell)& \mathrm{if\ }\epsilon(\ell)=-1\\
\\
\ord_\ell(n) R_{\fa\fq\fn\fl}(n/\ell)& \mathrm{if\ }\epsilon(\ell)=0
\end{array}\right.
\end{equation*}
where $\fn$ is any integral $\co_s$-ideal with $\co_s/\fn\iso\Z/N\Z$.
When $\epsilon(\ell)=-1$, we may take  $\N(\fq)\equiv -\ell\pmod{Dp}$,
and when $\epsilon(\ell)=0$ we may take $\N(\fq\fl)\equiv -\ell\pmod{Dp}$.
\end{Prop}

If $\hat{K}^\times$ denotes the group of finite ideles of $K$ and
$\hat{\co}_s^\times\subset \hat{K}^\times$ 
is the group of units in the profinite completion of $\co_s$,
then there is an action of the ring class group
$\hat{K}^\times/K^\times\hat{\co}_s^\times\iso\Pic(\co_s)$
on $\mathcal{S}$: if $x=(x_r)\in\hat{K}^\times$ and $S\in\mathcal{S}$
then $S^x$ is defined by the relation $(S^x)_r=x_r S_rx_r^{-1}\subset B_r$
for every rational prime $r$.  In terms of $\co_s$-ideals the action is
again by conjugation: $S^\fb=\fb S\fb^{-1}$.

\begin{Lem}\label{Eichler}
The action of  $\Pic(\co_s)$ on $\mathcal{S}$ is transitive,
and the stabilizer of any element is the subgroup generated by the 
class of $\fl$ (so has order $1$ if $\epsilon(\ell)=-1$ and order
$2$ if $\epsilon(\ell)=0$).
\end{Lem}
\begin{proof}
Let $S$ and $S'$ be $\co_s$-optimal level $N$ Eichler orders.  
To prove the transitivity of the action of $\Pic(\co_s)$
on $\mathcal{S}$, we must show that $S_r$ and $S'_r$ are conjugate
by elements of $K_r^\times$ for every prime $r$.
The proof of \cite[Theorem A.15]{mann} shows that this is the case
if either $\co_{s,r}$ is maximal (which occurs for all $r\not=p$)
or if $S_r$ and $S_r'$ are maximal (which occurs for all $(r,N)=1$).
To compute the kernel of the action,
fix $S\in\mathcal{S}$ and let $x=(x_r)$ be a finite idele of $K$.
If $S=S^x$ in $\mathcal{S}$ then there is some $y\in K^\times$
such that $x_ry_r^{-1}$ is contained in $N(S_r)$,
the normalizer of $S_r$ in $B_r^\times$,
for every prime $r$.  

If $(r,N\ell)=1$ then $N(S_r)=\Q_r^\times S_r^\times$,
and so $$x_r y_r^{-1}\in (\Q_r^\times S_r^\times)
\cap K_r^\times=\Q_r^\times\co_{s,r}^\times.$$
If $r|N$ then $\Q_r^\times S_r^\times$ has index $2$ in $N(S_r)$.  
Fix an isomorphism
$\psi:B_r\iso M_2(\Q_r)$ in such a way that $\psi(K_r)\iso\Q_r\oplus\Q_r$ 
is the quadratic subalgebra of diagonal matrices, 
and let $S_r'\subset M_2(\Q_r)$ be
the usual Eichler order of integral matrices whose lower left entry 
is divisible by $N_r=r^{\ord_r(N)}$.  
As $S_r$ and $\psi^{-1}(S_r')$ are both $\co_{s,r}$-optimal, 
by the discussion above there is a 
$z\in K_r^\times$ such that $z S_r z^{-1}= \psi^{-1}(S_r')$.
Thus replacing $\psi$ by a $\psi(K^\times)$-conjugate we 
may also assume that $\psi(S_r)=S_r'$. Having made such a choice, we now
supress $\psi$ from the notation.
The nontrival coset
of $\Q_r^\times S_r^\times$ in $N(S_r)$ is represented by the
matrix 
$\alpha=\begin{pmatrix}0 & 1\\ N_r&0 \end{pmatrix}$, and
one now checks directly
that 
$$
x_r y_r^{-1}\in N(S_r)\cap K_r^\times=(\Q_r^\times S_r^\times\sqcup
\alpha\Q_r^\times S_r^\times)
\cap K_r^\times=\Q_r^\times\co_{s,r}^\times.
$$
When $r=\ell$, $B_\ell$ has a unique maximal order,
hence $N(S_\ell)\cap K_\ell^\times=K_\ell^\times$. 
We have shown that a finite idele 
$(x_r)$ acts trivially on $\mathcal{S}$ if and only if
$(x_r)\in \hat{\Q}^\times\hat{\co}_s^\times K^\times_\ell K^\times
=\hat{\co}_s^\times K^\times_\ell K^\times$.
\end{proof}

Let $\mathcal{W}_0$ denote the set of prime divisors of $Dp$ if
$\epsilon(\ell)=-1$, and the set of prime divisors $\not=\ell$
of $Dp$ if $\epsilon(\ell)=0$. 
Let $\mathcal{W}$ be the free abelian group (written multiplicatively)
of exponent $2$ on the elements of $\mathcal{W}_0$,
and define a homomorphism
$$
\mathcal{W}\map{}\Pic(\co_s)[2]
$$
by sending $w\mapsto(\sqrt{D})_w$, the finite idele of 
$K$ which is $1$ away from $w$
and equal to the image of $\sqrt{D}$ under $K^\times\map{}K_r^\times$
at each $r|w$.  This map allows us to view $\mathcal{S}$ as a 
$\mathcal{W}$-module. By genus theory,
the map $\mathcal{W}\map{}\Pic(\co_s)[2]$ is surjective.  The
kernel has order $2$ if $\epsilon(\ell)=-1$, and has order $1$ if
$\epsilon(\ell)=0$.

As in \cite[pp. 265-266]{gz}, we now choose a particular model
for the quaternion algebra $B$.  Detailed proofs of the following
assertions can be found in \cite{mann}.
If $\epsilon(\ell)=-1$  then choose a prime $q$
such that $\left(\frac{-\ell q}{r}\right)=1$ for all primes $r\mid D$.
For such a $q$ the quaternion algebra $B$ is isomorphic
to the quaternion algebra
$\left(\frac{D,-\ell q}{\Q}\right)$ 
(meaning the quaternion algebra
$B=\Q\oplus\Q i\oplus\Q j\oplus \Q ij$ with $i^2=D$, $j^2=-\ell q$,
$ij=-ji$) and  $q$ is split in 
$K$.   We may, and do, further impose the condition $q\equiv -\ell\pmod{Dp}$.
If $\epsilon(\ell)=0$ then choose a prime $q\not=\ell $
such that $\left(\frac{-q}{r}\right)=1$
for all primes $r\mid (D/\ell)$, and with $\left(\frac{-q}{\ell}\right)=-1$.
For such a $q$ the quaternion algebra $B$ is isomorphic
to the quaternion algebra
$\left(\frac{D,-q}{\Q}\right)$, and again such a $q$ is split in 
$K$. We further impose the condition $\ell q\equiv -\ell\pmod{Dp}$.
We henceforth fix a $q$ as above and identify
$$
B\iso\left\{\begin{array}{ll}
\left(\frac{D,-\ell q}{\Q}\right) & \mathrm{if\ } \epsilon(\ell)=-1\\
\left(\frac{D,-q}{\Q}\right) & \mathrm{if\ } \epsilon(\ell)=0.
\end{array}\right.
$$
In either case we regard $K$ as a subfield of 
$B$ via $\sqrt{D}\mapsto i$, so that conjugation by 
$j$ acts as complex conjugation on $K$.  Let $\diff_s=p^s\sqrt{D}\co_s$
denote the different of the order $\co_s$.  Fix an integral $\co_s$-ideal
$\fn$ such that $\co_s/\fn\iso\Z/N\Z$, and let $\fq$ be an
integral $\co_s$-ideal of norm $q$.

\begin{Lem}
If $\epsilon(\ell)=-1$ there is a collection 
$\{X_r\in \Z_r^\times\mid r\in\mathcal{W}_0\}$
such that 
$$
R=\{\alpha+\beta j\mid \alpha\in\diff_s^{-1},\ \beta\in\diff_s^{-1}
\fn \fq^{-1},\ \alpha-X_r\beta\in \co_{s,r}\ \forall r\in\mathcal{W}_0\}
$$
is an $\co_s$-optimal level $N$ Eichler order, and such that 
$X_r^2= -\ell q$.
If $\epsilon(\ell)=0$ there is a collection 
$\{X_r\in\Z_{r}^\times\mid r\in\mathcal{W}_0\}$ 
such that
$$
R=\{\alpha+\beta j\mid \alpha\in\diff_s^{-1}\fl,\ \beta\in\diff_s^{-1}\fl
\fn\fq^{-1},\ \alpha-X_r\beta\in \co_{s,r}\ \forall r\in\mathcal{W}_0\}
$$ 
has the above property, and $X_r^2= -q$.  
\end{Lem}

\begin{proof}
Suppose $\epsilon(\ell)=-1$.
The order $S=\co_s+\fq^{-1} j\subset B$ has reduced discriminant
$p^{2s}D\ell$, and for a prime $r$ not dividing $pND$, $R_r=S_r$.
Thus the lattice $R_r$ is a maximal order at such primes.  If $r|N$
then $R_r=\co_{s,r}+\fn_r j$ is an Eichler order of level $r^{\ord_r N}$,
so it remains to consider $R_r$ for $r|Dp$.  We have assumed 
$q\equiv -\ell\pmod{Dp}$, so that by Hensel's lemma
$j^2=-\ell q$ has a square root
$X_r\in\Z_r^\times$ for each $r|Dp$.  If we set $t_r=X_r-j$ then one
readily computes $jt_r= -X_r t_r$, so that $B_r\cdot t_r=K_r\cdot t_r$ 
is a two-dimensional $\Q_r$-vector space
on which $B_r$ acts by left multiplication.  
Exactly as in the proof of Lemma \ref{unstable},
the (necessarily maximal) order leaving $\co_{s,r}\cdot t_r$ stable
is 
$$
R_r=\{\alpha+\beta j\in B_r\mid \alpha, \beta X_r\in \diff^{-1}_{s,r},
\alpha-\beta X_r\in\co_{s,r}\}.
$$
This shows that $R$ is a level $N$ Eichler order,
and the $\co_s$-optimality is immediate from the explicit description.
The case $\epsilon(\ell)=0$ is entirely similar.
\end{proof}

Fix a family $\{X_r\}$ and an order $R$ as in the lemma. It is
verified by direct calculation that for any $w\in\mathcal{W}$,
$R^w$ has the same explicit form as  $R$, but with $X_r$ replaced
by 
$$X_r^w=\left\{\begin{array}{ll}
-X_r & \mathrm{if\ }r|w\\
X_r & \mathrm{otherwise.}
\end{array}\right.$$

\begin{Lem}\label{prequaternion sum}
If $\fg$ is any integral $\co_s$-ideal of norm prime to $Dp$ then
\begin{eqnarray}\label{double ideals}\lefteqn{
\sum_{w\in\mathcal{W}}
\sum_{{b\in D^\fa_s(R^{w \fg},m)}}(1+\ord_\ell\N(b^-))= }\nonumber \\
& &
\sum_{ \genfrac{}{}{0pt}{}{n>0}{\ell|n,(n,p)=1} } 
\delta(n)r_\fa(mp^{2s}|D|-nN)\cdot
\left\{\begin{array}{ll}
4\cdot r_{\fa\fq\bar{\fn}\bar{\fg}^2}(n/\ell)\ \ord_\ell(\ell n) &  
\mathrm{if\ }\epsilon(\ell)=-1 \\
 2\cdot r_{\fa\fq\bar{\fn}\bar{\fg}^2\fl}(n/\ell)\ \ord_\ell(n) &
\mathrm{if\ }\epsilon(\ell)=0.
\end{array}\right.
\end{eqnarray}
\end{Lem}

\begin{proof}
Suppose that $\epsilon(\ell)=-1$.  The lattice $R^{w\fg}\fa$
is given explicitly by
$$R^{w\fg}\fa=\{\alpha+\beta j\mid \alpha\in\diff_s^{-1}\fa,\ 
\beta\in\diff_s^{-1}\fn \fq^{-1}\fg\bar{\fg}^{-1}\bar{\fa},
\ \alpha-X^w_r\beta\in \co_{s,r}\forall r|Dp\}.
$$
Denote by $\mathfrak{C}$ the set of all pairs $(\mathfrak{c}^+,
\mathfrak{c}^-)$ of proper, integral $\co_s$-ideals such that
\begin{enumerate}
\item $\N(\mathfrak{c}^+)+\ell N\N(\mathfrak{c}^-)=mp^{2s}|D|$,
\item $\mathfrak{c}^+$ and $\mathfrak{c}^-$ are prime to $p$,
\item $\mathfrak{c}^+$ lies in the $\Pic(\co_s)$-class of $\bar{\fa}$
\item $\mathfrak{c}^-$ lies in the $\Pic(\co_s)$-class of 
$\fa\bar{\fn}\fq\bar{\fg}^2$
\end{enumerate}
and for each $w\in\mathcal{W}$ let 
$F^w: D^\fa_s(R^{w\fg},m)\map{}\mathfrak{C}$
be the function defined by sending $b=\alpha+\beta j$ to the pair
\begin{equation}\label{quaternions to ideals}
\mathfrak{c}^+=\alpha\diff_s\fa^{-1}\hspace{1cm}\mathfrak{c}^-=
\beta\diff_s\fq\fn^{-1}\fg^{-1}\bar{\fg}\bar{\fa}^{-1}.
\end{equation}
If $D^\fa_s(R^{w\fg},m)$ contained both $b=\alpha+\beta j$ and 
$\alpha-\beta j$ then we would have $b^+=\alpha\in\co_{s,p}$,
contradicting $\ord_p\N(b^+)=-2s$.  This implies that $F^w$ is two-to-one.

The claim is that every element of $\mathfrak{C}$ is in the image of $F^w$
for exactly $2\delta(\N(\mathfrak{c}^-))$ choices of $w$, so that
\begin{equation}\label{double ideals II}
\sum_{w\in\mathcal{W}}\sum_{b\in D^\fa_s(R^{w \fg},m)}(1+\ord_\ell\N(b^-))=
4\sum_{(\mathfrak{c}^+,\mathfrak{c}^-)\in\mathfrak{C}}
(2+\ord_\ell\N(\mathfrak{c}^-))\cdot \delta(\N(\mathfrak{c}^-)).
\end{equation}
To verify this, fix $(\mathfrak{c}^+,\mathfrak{c}^-)\in\mathfrak{C}$
and choose generators
$$\alpha\co_s=\mathfrak{c}^+\diff_s^{-1}\fa
\hspace{1cm}
\beta\co_s=\mathfrak{c}^-\diff_s^{-1}\fq^{-1}\fn\fg\bar{\fg}^{-1}\bar{\fa}.
$$
Then $b=\alpha+\beta j$ lies in $D_s^\fa(R^{w\fg},m)$ if and only if
$\alpha-X_r^w\beta\in\co_{s,r}$ for every prime divisor $r$ of $Dp$,
or equivalently, if $\alpha'\equiv X_r^w\beta'\pmod{\diff_{s,r}}$ for 
every $r$, where 
$\alpha'=p^s\sqrt{D}\alpha,\ \beta'=p^s\sqrt{D}\beta\in\co_s$.
The action of complex conjugation on $\co_s/\diff_s$ is trivial
and so we have
$$\alpha'^2\equiv\N(\alpha')=\N(\fa)\N(\mathfrak{c}^+)
\equiv- \ell N\N(\mathfrak{c}^-)\N(\fa)=-\ell q\N(\beta')\equiv
X_r^2\beta'^2
$$
modulo $\diff_{s,r}$.  When $r\not=p$, $\co_{s,r}/\diff_{s,r}$
is a field, and so $\alpha'\equiv\pm X_r\beta'$.  The congruence
holds for both signs if and only if $\alpha'\equiv 0$, which holds
if and only if $r|\N(\mathfrak{c}^-)$. 
When $r=p$, $\alpha'\in\co_{s,r}^\times$ and the unit group of the ring 
$\Z/p^{2s}\Z\iso \co_{s,r}/\diff_{s,r}$ has no $2$-torsion 
apart from $\pm 1$.  Hence $\alpha'\equiv \pm X_r\beta'$
for a unique choice of sign.  We have shown that 
$\alpha+\beta j$ is contained in $D^\fa_s(R^{w\fg},m)$ for 
exactly $\delta(\N(\mathfrak{c}^-))$ choices of $w$.  
The element $\alpha-\beta j$ lies in $D^\fa_s(R^{w\fg},m)$ for another 
$\delta(\N(\mathfrak{c}^-))$ choices of $w$, all distinct from the 
first set of choices.  This proves (\ref{double ideals II}).
The right hand side of (\ref{double ideals II}) agrees with the right
hand sum in the statement of the lemma by setting
$n=\ell\N(\mathfrak{c}^-)$.

The case where $\epsilon(\ell)=0$ is very similar: the set 
$\mathfrak{C}$ is instead taken to be the collection of pairs 
of proper, integral $\co_s$-ideals
$(\mathfrak{c}^+,\mathfrak{c}^-)$ such that
\begin{enumerate}
\item $\N(\mathfrak{c}^+)+ N\N(\mathfrak{c}^-)=mp^{2s}|D|$,
\item $\mathfrak{c}^+$ and $\mathfrak{c}^-$ are prime to $p$
and divisible by $\fl$
\item $\mathfrak{c}^+$ lies in the $\Pic(\co_s)$-class of $\bar{\fa}$
\item $\mathfrak{c}^-$ lies in the $\Pic(\co_s)$-class of 
$\fa\bar{\fn}\fq\bar{\fg}^2$.
\end{enumerate}
The function from $D_s^w(R^{w\fg},m)$ to $\mathfrak{C}$ is then
exactly as in (\ref{quaternions to ideals}), and 
the expression on the left hand side of (\ref{double ideals}) 
is equal to
\begin{eqnarray*}\lefteqn{
4\sum_{(\mathfrak{c}^+,\mathfrak{c}^-)\in\mathfrak{C}}
\ord_\ell\N(\mathfrak{c}^-)\cdot 
2^{ \#\{ r\in\mathcal{W}_0\mid r\mathrm{\ divides\ } \N(\mathfrak{c}^-)  \} }
= }\hspace{2cm}\\
& & 2\sum_{ \genfrac{}{}{0pt}{}{n>0}{\ell|n,(n,p)=1 } }
r_\fa(mp^{2s}|D|-nN)r_{\fa\fq\bar{\fn}\bar{\fg}^2}(n)
\delta(n)\ord_\ell(n)
\end{eqnarray*}
by taking $n=\N(\mathfrak{c}^-)$. This is equivalent to the stated
equality.
\end{proof}

\begin{proof}[Proof of Proposition \ref{quaternion sum}]
Fix a set $\mathfrak{G}=\{\fg\}$ of proper integral
$\co_s$-ideals of norm prime to $Dp$ such that 
$\{\fg^2\mid \fg\in\mathfrak{G}\}$ represents $\Pic(\co_s)^2$.
As $\fg$ varies over $\mathfrak{G}$ and $w$ varies over $\mathcal{W}$,
$w\fg$ varies over $\Pic(\co_s)$ hitting each ideal class once
if $\epsilon(\ell)=0$ and twice if $\epsilon(\ell)=-1$.
By Lemmas \ref{Eichler} and \ref{prequaternion sum} (recall
also that we are assuming $D$ odd) we have
\begin{eqnarray*}
\Delta_s^\fa(m) &=&
\frac{1}{2}\sum_{w\in\mathcal{W}}\sum_{\fg\in\mathfrak{G}}
\Delta_s^\fa(R^{w\fg},m) \\
&=& \frac{1}{2(1-\epsilon(\ell))}
\sum_{\fg\in\mathfrak{G}}\sum_{w\in\mathcal{W}}
\sum_{b\in D_s^\fa(R^{w\fg},m)}\big(1+\ord_\ell\N(b^-)\big)\\
&=& \sum_{ \genfrac{}{}{0pt}{}{n>0}{\ell|n,(n,p)=1 } } 
\delta(n) r_\fa(mp^{2s}|D|-nN)\cdot
\left\{\begin{array}{ll}
\ord_\ell(\ell n) R_{\fa\fq\fn}(n/\ell)& \mathrm{if\ }\epsilon(\ell)=-1\\
\\
\ord_\ell(n) R_{\fa\fq\fn\fl}(n/\ell)& \mathrm{if\ }\epsilon(\ell)=0.
\end{array}\right.
\end{eqnarray*}
\end{proof}


\subsection{The $\ell$-contribution to the height}
\label{conjugation}


Fix $m=m_0p^r$ with  $(m_0,Np)=1$.
Let $\fb$ be a  proper integral $\co_s$-ideal, and denote by 
$\tau\in\Gal(H_s/K)$ the Artin symbol of $\fb$.
We consider the quantity
$$
\langle c_s^\tau, T_{m_0}(\bd_{s,r+2}^{\sigma\tau})\rangle_v 
-\langle c_{s-1}^\tau, T_{m_0}(\bd_{s,r+1}^{\sigma\tau})\rangle_v
$$
where the pairing is the local N\'eron symbol on $X_{/H_{s,v}}$
of Proposition \ref{curve height}.
By replacing $h_i$ with $h_i^\tau$ in Proposition \ref{intersections}, 
this is equal to
\begin{eqnarray*}\lefteqn{
\log_p(\ell)\sum_{n\ge 0}
\Big( |\Hom_{W_n}( {\uh}_s^{\fa\fb}, {\uh}_s^\fb)_{\deg(mp^2)}|
-|\Hom_{W_n}( {\uh}_{s-1}^{\fa\fb}, {\uh}^\fb_s)_{\deg(mp)}|
} \hspace{2cm}\\ 
& &-|\Hom_{W_n}( {\uh}_{s}^{\fa\fb}, {\uh}^\fb_{s-1})_{\deg(mp)}|
+|\Hom_{W_n}( {\uh}_{s-1}^{\fa\fb}, {\uh}^\fb_{s-1})_{\deg(m)}| 
\Big),
\end{eqnarray*}
which is equal, by Corollary \ref{IE}, to
\begin{eqnarray*}
\log_p(\ell)\sum_{n\ge 0}\Big(
|\Hom_{W_n}( \uh_s^{\fa\fb}, \uh_s^\fb)_{\deg(mp^2)}^\unstable|
-|\Hom_{W_n}( \uh_s^{\fa\fb}, \uh_s^\fb)_{\deg(m)}^\unstable|\Big).
\end{eqnarray*}
By Corollary \ref{unstable to quaternion},
this last expression is equal to 
$$ \log_p(\ell)\Big(
\Delta_s^\fa(R^{\fb^{-1}},mp^2)-\Delta_s^\fa(R^{\fb^{-1}},m)\Big),
$$
where we have used \cite[(7-8)]{conrad} to identify 
$\End_{W_0}(\uh_s^\fb)$ with $\fb^{-1}\cdot \End_{W_0}(\uh_s)\cdot \fb$ inside
of $B=\Hom_{W_0}(\uh_s,\uh_s)\otimes\Q$.

\begin{Prop}\label{nonsplit evaluation}
For any positive integer $m=m_0p^r$ with $(m_0,Np)=1$ and
any $\ell$ nonsplit in $K$,
\begin{eqnarray*}
\sum_w\Big(
\langle c_s, T_{m_0}(\bd_{s,r+2}^{\sigma})\rangle_w
-\langle c_{s-1}, T_{m_0}(\bd_{s,r+1}^{\sigma})\rangle_w \Big)
&= &\log_p(\ell)\big(\Delta_s^\fa(mp^2)-
\Delta_s^\fa(m)\big)
\end{eqnarray*}
where the sum is over all primes $w$ of $H_s$ above $\ell$ and
$\Delta_s^\fa(m)$ is the quantity defined in \S \ref{quaternions}
(and computed explicitly in Proposition \ref{quaternion sum}),
and the pairing is the local N\'eron symbol on $X_{/H_{s,w}}$
of Proposition \ref{curve height}.
\end{Prop}

\begin{proof}
Let $\Pic^\ell(\co_s)$ denote the quotient of $\Pic(\co_s)$ by
the subgroup generated by the class of the 
unique prime of $K$ above $\ell$.  Then
$\Pic^\ell(\co_s)$ acts simply transitively on the set $\mathcal{S}$
by Lemma \ref{Eichler}, and also acts simply transitively on the 
primes of $H_s$ above $\ell$.  If we let $\fb$ vary over a set
of representatives of $\Pic^\ell(\co_s)$ and use the
relation $\langle x^\tau, y^\tau\rangle_v=
\langle x,y\rangle_{\tau^{-1}(v)}$ for $\tau\in\Gal(H_s/K)$,
then the claim follows from the discussion above.
\end{proof}


\section{N\'eron symbols above $p$}
\label{p neron}


In this section we use the methods of Perrin-Riou \cite[\S 5.3]{pr1}
to analyze the $p$-adic N\'eron symbol on $X_0(N)$ at primes above 
$p$.

Fix $s>0$, $\sigma\in\Gal(H_s/K)$, and assume that $\epsilon(p)=1$
and $D\not= -3,-4$.
As always, we let $\fa\subset\co_s$ be a proper ideal whose
Artin symbol is $\sigma$.
For any positive integer $m$, we let $T_m$ be the usual Hecke 
correspondence on $X_0(N)$ (taking the Atkin-Lehner $U_\ell$ at primes
dividing $N$).
For any correspondence $T$ from a curve to itself,
we let $T^\iota$ denote the transpose correspondence.
Thus $T_m=T_m^\iota$ for $(m,N)=1$.
If $\mathfrak{p}$ is one of the two primes of $K$ above $p$,
we let $\delta$ be the order of $\mathfrak{p}$ in the ideal class
group of $K$.


\subsection{Some modular forms}
\label{modular forms}


Fix a place $v$ of $H_s$ above $p$.

\begin{Lem}\label{delta congruence}
Let $R$ be the integer ring of $H_{s,v}$ and let 
$\ubh_{s,r}^\sigma$ be the horizontal divisor
of $X_0(N)_{/R}$ with generic fiber $\bh_{s,r}^\sigma$.
For any divisor $\uC$ on $X_0(N)_{/R}$, there is a constant $c=c(\uC)$
such that the intersection multiplicity 
$i(\uC,\ubh_{s,r}^\sigma)$ of \S \ref{intersection theory}
depends only on $r\pmod{\delta}$ when $r >c $.
\end{Lem}

\begin{proof}
It suffices to prove this when $\uC$ is effective.
The extension $H_\infty/H_0$ is totally ramified at $v$, and we let
$w$ denote the unique place of $H_\infty$ above $v$.
Let $F(r)$ be the completion of the maximal unramified extension of 
$H_{s+r,w}$ with integer ring $W(r)$, and let $W(r)_k$ be the quotient
of $W(r)$ by the $(k+1)^\mathrm{st}$-power of the maximal ideal.
Let $\hat{\Q}_p^\mathrm{unr}$
denote the completion of the maximal unramified extension 
of $\Q_p$.  The extension $H_{s+r,w}/H_{0,w}$ is totally ramified
of degree $p^{r+s-1}(p-1)$, and $H_{0,w}\subset \hat{\Q}_p^\mathrm{unr}$.
From this one easily deduces that $F(r)$ is the compositum of
$\hat{\Q}_p^\mathrm{unr}$ and  $H_{s+r,w}$ (so is abelian over $\Q_p$), and
that $F(r)/\hat{\Q}_p^\mathrm{unr}$ is totally ramified of degree
$p^{r+s-1}(p-1)$. By class field theory
$F(r)=\hat{\Q}_p^\mathrm{unr}(\mu_{p^{s+r}})$.
Decompose $\underline{C}=\sum_{k=0}^{e_r}\uy(k)$ 
as a sum of prime divisors on 
$X_0(N)_{/W(r)}$.  For $r$ greater than or equal to some $r_0$
the sequence $e_r$ is constant and $\uh_{s,r}^\sigma$ has no components
in common with $\uC$.  Abbreviate $e=e_{r_0}$ and take $c=r_0+\delta$.

Fix $r_1>c$, $r=r_1+i\delta$ with $i\ge 0$,
and an extension of $\sigma$ to $\Gal(H_\infty/K)$.
By \cite[Lemma 2.4]{conrad} or \cite[Theorems 8, 9(1)]{st}
the point $h_{s+r}^\sigma\in X_0(N)(F(r))$ 
represents a Heegner diagram over $F(r)$
having good reduction, and so its Zariski closure 
$\uh_{s+r}^\sigma$ in  $X_0(N)_{/W(r)}$ is a section to the structure map
representing a Heegner diagram over $W(r)$.
As in \S \ref{modular intersections}, the choice of Heegner diagram
$\uh_{s+r}^\sigma$ determines a family of
isogenous Heegner diagrams over $W(r)$, 
$$
\uh_{s+r}^\sigma\map{} \uh_{s+r-1}^\sigma\map{}\ldots.
$$
The generic geometric kernel of the map 
$\uh^\sigma_{s+r}\map{}\uh^\sigma_{s+r-1}$ is stable
under the action of the absolute Galois group of $F(r)$,
and the Euler system relations of \S \ref{Hecke action} tell us that no
other order $p$ subgroup of $\uh^\sigma_{s+r}(F(r)^\alg)$ has this property.
Indeed, the remaining $p$ quotients by order $p$ subgroups are
permuted simply transitively by $\Gal(F(r+1)/F(r))$.
It follows that this kernel must be the kernel in $\uh_{s+r}^\sigma[p]$
 of reduction to  $W(r)_0$ (recall $\epsilon(p)=1$, 
so $\uh^\sigma_{s+r}$ has ordinary reduction) and the map
$\uh^\sigma_{s+r}\map{}\uh^\sigma_{s+r-1}$
reduces to the absolute Frobenius in the closed fiber.
The action of $\co_{s+r}$ on the closed fiber of $\uh^\sigma_{s+r}$ 
extends to an action of the maximal order 
(we have just shown that the closed fiber of $\uh^\sigma_{s+r}$
is isomorphic to a Galois conjugate of the closed fiber of 
$\uh^\sigma_{0}$), 
and if $\mathfrak{p}$ denotes the prime of $K$ below $v$, then 
the action of any generator of the principal ideal $\mathfrak{p}^\delta$ is a 
degree $p^\delta$ purely inseparable endomorphism, whose kernel
must therefore be the kernel of the $\delta^\mathrm{th}$-iterate of
Frobenius.  This shows that the Heegner diagrams 
$\uh^\sigma_{s+r}$ and $\uh^\sigma_{s+r-\delta}$ are isomorphic
over $\Spec(W(r)_0)$, and that the closed fiber of $\uh_{s+r}^\sigma$
is the base change to $W(r)$ of the closed fiber of the
Zariski closure of $h_{s+r_1}^\sigma$
on $X_0(N)_{/W(r_1)}$.

We claim that the Heegner diagram
$\uh^\sigma_{s+r-\delta}$
is distinct from $\uh^\sigma_{s+r}$ over $W(r)_1$, so that
Proposition \ref{first homs} gives the intersection formula
\begin{equation}\label{transverse reductions}
i(\uh^\sigma_{s+r},\uh^\sigma_{s+r-\delta})=\frac{1}{2}|\co_K^\times|=1
\end{equation}
on $X_0(N)_{/W(r)}$.  Indeed, if these Heegner diagrams are
isomorphic over $W(r)_1$, then the reduction of such an isomorphism
to $W(r)_0$ allows us to view $\uh^\sigma_{s+r-\delta}$
and $\uh^\sigma_{s+r}$ over $W(r)_1$ 
as isomorphic \emph{deformations} of the 
common closed fiber, which we denote by $g$.  Let 
$T=\mil g(W(r)_0)[p^k]\iso\Z_p$. The theory of
Serre-Tate coordinates (for example \cite{Goren} Chapter 3, Theorem 4.2)
associates to these Heegner diagrams over $W(r)$ (viewed as deformations
of $g$) two bilinear maps 
$$q_{s+r-\delta},q_{s+r}:T\otimes T\map{}1+\mathfrak{m}_{W(r)}.$$
The first surjects onto $\mu_{p^{s+r-\delta}}$, and the
second onto $\mu_{p^{s+r}}$.  Since we assume
the Heegner diagrams over $W(r)_1$ are isomorphic 
as deformations of $g$,  the bilinear maps
$q_{s+r-\delta},q_{s+r}$ are congruent modulo
$1+\mathfrak{m}_{W(r)}^2$.
This is a contradiction, as
$\mu_{p^{s+r-\delta}}$ is contained in $1+\mathfrak{m}_{W(r)}^2$
while $\mu_{p^{s+r}}$ is not (use the fact, noted above,
that $F(r)=\hat{\Q}_p^\mathrm{unr}(\mu_{p^{r+s}})$ to replace
$\mathfrak{m}_{W(r)}$ with the maximal ideal of $\Z_p[\mu_{p^{r+s}}]$).

Each prime divisor $\uy(k)$ occuring in the support of $\uC$
either does not meet the common closed point of 
$\uh^\sigma_{s+r-\delta}$, $\uh^\sigma_{s+r}$,
in which case 
$i(\uy(k),\uh^\sigma_{s+r}) =0,$
or it does, in which case $\uy(k)$ intersects both 
$\uh^\sigma_{s+r-\delta}$ and $\uh^\sigma_{s+r}$.
Assume we are in the latter case.  The divisors $\uy(k)$ and 
$\uh_{s+r-\delta}^\sigma$
on $X_0(N)_{/W(r)}$ both arise as the base change of divisors defined over
$W(r-\delta)$.  Since base change through a finite extension multiplies 
intersections by the ramification degree,
$i(\uy(k),\uh^\sigma_{s+r-\delta})>1 $.
If also $i(\uy(k),\uh^\sigma_{s+r})>1$ then 
$i(\uh^\sigma_{s+r},\uh^\sigma_{s+r-\delta})>1$, 
contradicting (\ref{transverse reductions}).  
Thus $i(\uy(k),\uh^\sigma_{s+r})=1$. We have shown that
$$
i(\underline{C}, \ubh_{s,r}^\sigma)_R=
i(\underline{C}, \uh_{s+r}^\sigma)_{W(r)}
=\sum_{k=0}^{e}i(\uy(k),\uh^\sigma_{s+r})_{W(r)}
$$
(the subscripts denoting the bases over which the intersections are
computed) is equal to the number of $\uy(k)$, $0\le k\le e$, which contain the
closed point of $\uh^\sigma_{s+r}$.  By the discussion earlier this is 
equal to the number of $\uy(k)$ on $X_0(N)_{/W(r_1)}$ which contain the
closed point of the Zariski closure of $h^\sigma_{s+r_1}$
on $X_0(N)_{/W(r_1)}$, which is equal to 
$i(\underline{C}, \ubh_{s,r_1}^\sigma)_R$
by taking $r=r_1$ in the preceeding argument.
\end{proof}

Let us say that a  divisor $C$ on $X_0(N)_{/H_{s,v}}$ 
has \emph{good support} if its support contains no cusps except
possibly for the cusp $0$.  Note that the set of such divisors
is stable under the action of $T_m^i$ for any $m$.
This follows easily from the fact that the main Atkin-Lehner involution 
$w$ on $X_0(N)$ satisfies $wT_m w=T_m^\iota$ and $w\cdot \infty=0$, 
and that $T_m\cdot \infty$ is supported at $\infty$.
For $C$ of degree zero with good support we define a formal $q$-expansion
\begin{equation}\label{phi def}
\phi(C)_v=\sum_{m=m_0 p^r}
\langle C,T_{m_0}\bd^\sigma_{s,r}\rangle_v q^m\\
\end{equation}
where $\langle\ ,\ \rangle_v$ is the $p$-adic N\'eron symbol on 
$X_0(N)_{/H_{s,v}}$ of Proposition \ref{curve height},
and where for any integer $m>0$ we write
$m=m_0 p^r$ with $(m_0,p)=1$.
Let $U$ denote the shift operator on formal $q$-expansions
$U(\sum a_m q^m)=\sum a_{mp}q^m$.  The $q$-expansion $\phi(C)_v$
is only defined if $C$ has support prime to
$T_{m_0}(\bd^\sigma_{s,r})$ for every $m=m_0 p^r$, 
but for any $C$ with good support and degree $0$ 
the $q$-expansion $U^k\phi(C)_v$
is defined for $k\gg 0$.  Indeed, the geometric points
in the support of  $T_{m_0}(\bd_{s,r+k}^\sigma)$ each represent either
the cusp $\infty$ or a
CM elliptic curve such that the valuation at $p$ of the 
conductor of the CM order is exactly $s+r+k$.

We can use the Lemma \ref{delta congruence}
to compute $p$-adic N\'eron symbols at $v$
in the only case where they are known to be related to intersection pairings:
the case where one divisor is principal.

\begin{Cor}\label{principal neron}
Suppose  $C$ is the divisor of a rational function 
on $X_0(N)_{/H_{s,v}}$, and that $C$ has good support.
Then for each integer $m>0$
$$
\lim_{k\to\infty} a_m\big(U^{k}(U^\delta-1)\phi(C)_v\big)=0.
$$
\end{Cor}

\begin{proof}
Write $m=m_0 p^r$ with $(m_0,p)=1$. The divisor $T_{m_0}^\iota(C)$
is again principal with good support, and we fix a rational
function $f$ with $(f)=T_{m_0}^\iota(C)$.
Writing $v$ for the normalized valuation on $H_{s,v}$, the 
intersection theory of \S \ref{intersection theory} gives
$$
v\big(f(\bd^\sigma_{s,r+k+\delta})\big)=[(f),\bd^\sigma_{s,r+k+\delta}]
=i(\underline{(f)}, \ubh_{s,r+k+\delta}^\sigma))
-p^{r+k+\delta}\cdot i(\underline{(f)}, \underline{\infty}))
$$
where the underlining of divisors indicates passing to horizontal divisors
on $X_0(N)_{/R}$, $R$ the integer ring of $H_{s,v}$.
Similarly
$$
v\big(f(\bd^\sigma_{s,r+k})\big)=[(f),\bd^\sigma_{s,r+k}]
=i(\underline{(f)}, \ubh_{s,r+k}^\sigma))
-p^{r+k}\cdot i(\underline{(f)}, \underline{\infty})).
$$
From this and Lemma \ref{delta congruence} we deduce
\begin{eqnarray*}
v\left(\frac{f(\bh^\sigma_{s,r+k+\delta})}{f(\bh^\sigma_{s,r+k})}\right)
&=&
v\left(\frac{f(\bd^\sigma_{s,r+k+\delta})}{f(\bd^\sigma_{s,r+k})}\right)
+ (p^\delta-1)p^{r+k}\cdot v\big(f(\infty)\big) \\
&=&
(p^\delta-1)p^{r+k}\cdot \big[v\big(f(\infty)\big)-
i(\underline{(f)}, \underline{\infty}) \big]
\end{eqnarray*}
for $k$ large.  Multiplying $f$ by an element of $H_{s,v}^\times$ 
does not change
$(f)$, and so we may assume that 
$v(f(\infty))=i(\underline{(f)}, \underline{\infty})$.
Then $\frac{f(\bh^\sigma_{s,r+k+\delta})}{f(\bh^\sigma_{s,r+k})}$
is a unit in $H_{s,v}$ for $k$ large.
It is also the norm of some  $u_k\in H_{s+r+k,v}$, the completion of
$H_{s+r+k}$ at the unique prime above $v$.
Using Proposition \ref{curve height}(b)
\begin{eqnarray*}
a_m( U^k(U^\delta-1)\phi(C)_v )&=&
\langle C, T_{m_0}\bd^\sigma_{s,r+k+\delta}\rangle_v
-\langle C, T_{m_0}\bd^\sigma_{s,r+k}\rangle_v\\
&=&\rho_{H_{s,v}}(f(\bd^\sigma_{s,r+k+\delta}))-
\rho_{H_{s,v}}(f(\bd^\sigma_{s,r+k}))\\
&=& \rho_{H_{s,v}}
\left(\frac{f(\bh^\sigma_{s,r+k+\delta})}{f(\bh^\sigma_{s,r+k})}\right)
-(p^\delta-1)p^{r+k}\rho_{H_{s,v}}(f(\infty))\\
&=&\rho_{\Q_p}(\Norm_{H_{s+r+k,v}/\Q_p}(u_k))
-(p^\delta-1)p^{r+k}\rho_{H_{s,v}}(f(\infty)).
\end{eqnarray*}
\emph{Since $p$ is split}, the field $H_{s+r+k,v}$ is abelian 
over $\Q_p$, the unit norms from $H_{s+r+k,v}$ to $\Q_p$
converge to $1$ as $k\to\infty$, and so the final expression converges
to $0$.
\end{proof}

Given any point $P\in J_0(N)(H_{s,v})$ we may choose a degree zero
divisor $C$ on $X_0(N)_{/H_{s,v}}$ having good support
 which represents $P$. Corollary \ref{principal neron}
implies that for any sequence of integers $b=(b_k)$ with $b_k\to\infty$, 
the $q$-expansion with $\Q_p$-coefficients
$$
\Phi_b(P)_v\stackrel{\mathrm{def}}{=}
\lim_{k\to\infty} U^{b_k}(U^\delta-1)\phi(C)_v,
$$
if the limit exists (in the sense of coefficient-by-coefficient
convergence; there is no assuption of uniformity)
depends only on $P$ and not on the choice of $C$.

\begin{Def}\label{admissible}
A sequence of integers $b=(b_k)$ is \emph{admissible} if $b_k\to\infty$
and if the limit (coefficient-by-coefficient)
defining $\Phi_b(P)_v$ exists for every $P\in J_0(N)(H_{s,v})$.
\end{Def}

\begin{Lem}\label{admiss lemma}
Any sequence of integers tending to $\infty$ admits an admissible
subsequence.  
\end{Lem}

\begin{proof}
Fix a sequence $b=(b_k)$ of integers tending to $\infty$.
Let $C$ be a degree zero divisor on $X_0(N)_{H_{s,v}}$ with good support, and
consider the first Fourier coefficient
$$a_1(U^{b_k}(U^\delta-1)\phi(C)_v)=
\langle C, \bd^\sigma_{s,b_k+\delta}
-\bd^\sigma_{s,b_k}\rangle_v.$$
By the final claim of Proposition \ref{curve height} 
the sequence on the right hand side takes
values in a compact subset of $\Q_p$, and so we may choose a convergent
subsequence.  By Corollary \ref{principal neron} and the finite
dimensionality of $J_0(N)(H_{s,v})\otimes\Q_p$, 
we may repeat this process, eventually replacing $b$ by a 
subsequence (still denoted $b$, abusively) such that
$$\lim_{k\to\infty} a_1(U^{b_k}(U^\delta-1)\phi(C)_v)$$
exists for every degree zero divisor with good support.
By the same argument we may assume that the limit 
$\lim_{k\to\infty} a_p(U^{b_k}(U^\delta-1)\phi(C)_v)$
also exists for all such divisors.
Now fix $m=m_0p^r$ with $(m_0,p)=1$. From the definition of $\phi$ we have
\begin{equation}\label{recursive crap}
a_{m}(U^{b_k}(U^\delta-1)\phi(C)_v)
=
a_{p^r}(U^{b_k}(U^\delta-1)
\phi(T_{m_0}^\iota C)_v)
\end{equation}
(for $k$ large enough that both sides are defined).
If $r=0$ or $1$ then the limit as $k\to\infty$ exists by the above
choice of $b$.  For $r>1$ we use the Euler system relations of
\S \ref{Hecke action} to see that
\begin{eqnarray*}
\bd^\sigma_{s,r+b_k} 
&=&
\Norm_{H_{s+b_k+1}/H_s}\bd^\sigma_{s+b_k+1,r-1}\\
&=&
\Norm_{H_{s+b_k+1}/H_s}\big( T_{p^{r-1}}d^\sigma_{s+b_k+1}- 
T_{p^{r-2}}d^\sigma_{s+b_k} \big) \\
&=& T_{p^{r-1}} \bd^\sigma_{s,b_k+1} -p T_{p^{r-2}} \bd^\sigma_{s,b_k}
\end{eqnarray*}
which, together with the same formula with $b_k$ replaced by
$b_k+\delta$, implies that the right hand side of 
(\ref{recursive crap}) equals (for $k\gg 0$)
$$
a_p(U^{b_k}(U^\delta-1)\phi(T^\iota_{m_0 p^{r-1}}C)_v)-p\cdot
a_1(U^{b_k}(U^\delta-1)\phi(T^\iota_{m_0 p^{r-2}}C)_v),
$$
and this limit exists as $k\to\infty$.
\end{proof}

Fix an admissible sequence $b$.
Note that the above proof shows that 
\begin{equation}\label{phi recursion}
a_{mp}(\Phi_b(P)_v)=\left\{\begin{array}{ll}
a_p(\Phi_b(T_m^\iota P)_v) &\mathrm{if\ }(m,p)=1\\
a_p(\Phi_b(T_m^\iota P)_v)
-p a_1(\Phi_b(T_{m/p}^\iota P)_v)&\mathrm{else.}\end{array}\right.
\end{equation}
Let $\Hecke^\mathrm{full}$ denote the 
$\Q_p$-algebra generated by the Hecke operators
$T_m$ for all $m>0$ acting on $J_0(N)$.  
For any $P\in J_0(N)(H_{s,v})$ and any $i>0$,
the linear functional on $\Hecke^\mathrm{full}$ defined by
$T\mapsto a_i(\Phi_b(T^\iota P)_v)$
determines a $p$-adic modular form 
$$
h_i(P)=\sum a_i(\Phi_b(T_m^\iota P)_v)\cdot  q^m \in 
S_2(\Gamma_0(N),\Q)\otimes\Q_p
$$
of level $\Gamma_0(N)$ (as does any linear functional on 
$\Hecke^\mathrm{full}$; this follows from 
\cite[\S 5.3 Theorem 1]{Hida book} and the identification of 
$\Hecke^\mathrm{full}$ with the Hecke algebra acting on weight two
cusp forms).  
The relation (\ref{phi recursion}) can be written as 
$$
U\cdot \Phi_b(P)_v=h_p(P)-p V\cdot h_1(P)
$$
where $V(\sum a_n q^n)=\sum a_n q^{pn}$.  As $V$ takes modular forms
of level $\Gamma_0(N)$ to modular forms of level $\Gamma_0(Np)$, 
we may define
$$
\Psi_b(P)_v=U\cdot \Phi_b(P)_v\in M_2(\Gamma_0(Np),\cA)\otimes_{\cA}\cB
$$
for any $P\in J_0(N)(H_{s,v})$.


\subsection{Annihilation of $E_\sigma$}


Recall Hida's ordinary projector 
$e^\ord=\lim_{k\to\infty}U^{k!}$ from \S \ref{L function}.
Fix an admissible 
(in the sense of Definition \ref{admissible}, and for all primes above
$p$ simultaneously) subsequence $b=(b_k)$ of
$k!$ and define, for any $P\in J_0(N)(H_s)$, a $p$-adic modular form
$\Psi_b(P)=\sum_{v|p}\Psi_b(P)_v$ where the
sum is over primes $v$ of $H_s$ above $p$.  Similarly,
define $\phi(C)=\sum_v\phi(C)_v$ (whenever $\phi(C)_v$ is
defined for all $v$ above $p$).

In the next section we shall see that there is a modular form
$$E_\sigma\in M_2(\Gamma_0(Np^\infty),\cA)\otimes\cB$$ with the following
property: if  $\langle\ ,\ \rangle_p$
denotes the sum of the local $p$-adic N\'eron symbols on $X_0(N)_{/H_{s,v}}$
at the primes of $H_s$ above $p$, then for any $m=m_0p^r$ with $(m_0,Np)=1$
the $m^\mathrm{th}$ Fourier coefficient of $E_\sigma$
is given by the expression
\begin{eqnarray*}
a_m(E_\sigma)&=&\langle c_s, T_{m_0}(\bd^\sigma_{s,r+2})\rangle_p-
\langle c_{s-1}, T_{m_0}(\bd^\sigma_{s,r+1})\rangle_p\\
&=& a_{mp^2}(\phi(c_s))-a_{mp}(\phi(c_{s-1})),
\end{eqnarray*}
where, as in \S \ref{the proof}, $c_i=(h_i)-(0)$.
From this we immediately deduce the following

\begin{Lem}\label{E decomp}
There is a  modular form $g\in M_2(\Gamma_0(Np),\cA)\otimes\cB$
such that $a_m(g)=0$ whenever $(m,N)=1$, and
$$(U^\delta-1)e^\ord E_\sigma=U\Psi_b(c_s)-\Psi_b(c_{s-1})+g.$$
\end{Lem}

\begin{proof}
Compare both sides coefficient-by-coefficient.
\end{proof}

The significance of Lemma \ref{E decomp} is the following:
while $E_\sigma$ depends a priori on the
divisors $c_s$ and $c_{s-1}$, the $p$-adic modular forms $\Psi_b(c_s)$ and 
$\Psi_b(c_{s-1})$ depend only on the images in $J_0(N)(H_s)$.  
This plays a  crucial role in the
proof of the following proposition.

\begin{Prop}\label{annihilation}
Let $f$ be the modular form fixed in the introduction.
The $p$-adic modular form $E_\sigma$ is annihilated by the linear functional
$L_f$ of Lemma \ref{linear lemma}.
\end{Prop}

\begin{proof}
By Lemmas \ref{linear lemma}(c,d) and \ref{E decomp} 
$$
(\alpha^\delta-1)L_f(E_\sigma)=L_f((U^\delta-1)e^\ord E_\sigma)
=\alpha L_f(\Psi_b(c_s))- L_f(\Psi_b(c_{s-1})),
$$
and so it suffices to show that
$L_f(\Psi_b(P)_v)=0$ for every $P\in J_0(N)(H_s)$ and every prime $v$ of
$H_s$ above $p$.  Fix one such prime and let $\Hecke$ 
be the $\Q$-algebra generated by all
$T_\ell$ with $(\ell,N)=1$ acting on $J_0(N)$.
Recall from the introduction the decomposition 
$$
J_0(N)(H_s)\otimes\cB\iso
\bigoplus_{\beta}J(H_s)_{\beta}
$$
where the sum is over all algebra homomorphisms
$\beta:\Hecke\map{}\Q_p^\alg$ (and recall that all such maps take values in 
$\cB$ by hypothesis) and 
$\Hecke$ acts on $J(H_s)_{\beta}$ through the character $\beta$.
Let $\beta_f$ be the homomorphism associated to the fixed newform $f$.

Suppose $P\in J(H_s)_\beta$ for some character $\beta$, and
extend $\Psi_b(\ )_v$ $\cB$-linearly to $J_0(N)(H_{s})\otimes\cB$.
We treat the cases $\beta\not=\beta_f$ and $\beta=\beta_f$ separately.

\begin{Lem}\label{Ann II}
If $\beta\not=\beta_f$ then $L_f(\Psi_b(P)_v)=0$.
\end{Lem}

\begin{proof}
Use the notation $\tilde{T}_m$ for 
Hecke operators in level $\Gamma_0(Np)$.
For any $m$ prime to $Np$ we have
$$
a_m(f) L_f(\Psi_b(P)_v)= L_f(\tilde{T}_m\Psi_b(P)_v)
=L_f(\Psi_b(T_m P)_v)
=\beta(T_m)L_f(\Psi_b(P)_v)
$$
(the first equality is by Lemma \ref{linear lemma}, the
second is a straightforward calculation, and the third is obvious).
Thus if $L_f(\Psi_b(P)_v)\not=0$ then $\beta_f(T_m)=\beta(T_m)$ 
for all $(m,Np)=1$. The Atkin-Lehner strong multiplicity one theorem 
\cite[Lemma 24]{AL} thus implies that $\beta_f=\beta$, a contradiction.
\end{proof}

\begin{Lem}
If $\beta=\beta_f$ then $L_f(\Psi_b(P)_v)=0$.
\end{Lem}

\begin{proof}
We follow the lead of \cite[Exemple 4.12]{pr1}.
Let $R$ be the integer ring of $H_{s,v}$,
$\mathfrak{m}$ the maximal ideal of $R$, and $\mathbf{F}=R/\mathfrak{m}$.
Let $G_n$ be the $p^n$-torsion
of the  N\'eron model of $J_0(N)$ over $R$, a finite group scheme over $R$.
Let $G_n^0$ and $G_n^{\mathrm{et}}$ be the connected component
and maximal \'etale quotient of $G_n$, respectively, and let
$G_n^{0,\mathrm{et}}$ (resp. $G_n^{0,0}$) be the maximal subgroup scheme
of $G_n^0$ with \'etale dual (resp. quotient with connected dual).

By the theory of Dieudonn\'e modules the Frobenius and Verschiebung
morphisms on $(G^{0,0}_n)_{/\mathbf{F}}$ are nilpotent, and so by
the Eichler-Shimura congruence the same is true of the Hecke
operator $T_p$.  This is equivalent to 
$T_p^i(I)\subset\mathfrak{m} I$ for some $i$, where
$A$ is the Hopf algebra over $R$ 
associated to the affine group scheme $G^{0,0}_n$, $I$ is the kernel of the 
augmentation map $A\map{}R$, and $T_p$ is now viewed as an $R$-algebra
map $A\map{}A$.  For any Artinian quotient
$R/\mathfrak{m}^k R$ of $R$ and any $R$-algebra map 
$\tau:A\map{} R/\mathfrak{m}^k R$,
$$
(\tau\circ T_p^{ik})(I)\subset \tau(\mathfrak{m}^kI)=0.
$$
Back in the world of group schemes, this says that $T_p$ acts as a nilpotent
operator on $G_n^{0,0}(R/\mathfrak{m}^k)$ for any $k$ and any $n$.
From this it follows easily that $T_p$ acts as a topologically
nilpotent operator on $R$-valued points of the 
formal group scheme $\hat{G}^{0,0}$
associated to the $p$-divisible group $\dlim G_n^{0,0}$.

Let $\hat{G}^0$ and $\hat{G}^{0,\mathrm{et}}$ be the
formal group schemes associated to $G_n^0$ and $G_n^{0,\mathrm{et}}$,
respectively.  As $\hat{G}^0(R)\subset J_0(N)(H_{s,v})$ with 
finite index, we may identify 
$$
\hat{G}^0(R)\otimes\cB\iso J_0(N)(H_{s,v})\otimes\cB.
$$
As $\beta_f(T_p)=a_p(f)\in\cA^\times$ is a unit, any element of
$\hat{G}^0(R)\otimes\cB$
 on which $\Hecke$ acts through $\beta_f$ must come
from the subspace $\hat{G}^{0,\mathrm{et}}(R)\otimes_{\Z_p}\cB$.
We are thus reduced to the case $P\in \hat{G}^{0,\mathrm{et}}(R)$.
By \cite[Theorem 1(i)]{Schneider} (together with the proof of
\cite[Theorem 2]{Schneider}), the universal norms in 
$\hat{G}^{0,\mathrm{et}}(R)$ from any ramified $\Z_p$-extension
of $H_{s,v}$ have finite index.  We are thus further reduced 
to the case where $P\in J_0(N)(H_{s,v})$ is a universal 
norm from $L_\infty$, the cyclotomic $\Z_p$-extension of $H_{s,v}$.
Let $L_n\subset L_\infty$ be the extension of $H_{s,v}$ with 
$[L_n:H_{s,v}]=p^n$, and write
$P=\N_{L_n/L_0}Q_n$ for some $Q_n\in J_0(N)(L_n)$.
Lift $Q_n$ to a degree zero divisor on $X_0(N)_{/L_n}$ with support prime to
the cusps.
Then for $m=m_0p^r$ with $(m_0,p)=1$,
\begin{eqnarray*}
a_{m}(\Psi_b(P)_v) &=&
\lim_{k\to\infty}
a_m\big(U^{b_k+1}(U^\delta-1)\phi(\N_{L_n/L_0}Q_n)_v\big) \\
& = & \lim_{k\to\infty}
\big\langle \N_{L_n/L_0}Q_n , 
T_{m_0}\bd^\sigma_{s,b_k+1+\delta+r}
-T_{m_0}\bd^\sigma_{s,b_k+1+r}
\big\rangle_{X_0(N),H_{s,v}}.
\end{eqnarray*}
Using Proposition \ref{curve height}(e), we at last deduce $\Psi_b(P)_v=0$.
\end{proof}

This completes the proof of Proposition \ref{annihilation}.
\end{proof}

\begin{Rem}
The reader is invited to reconsider the case $\beta=\beta_f$ 
under the additional hypothesis that $f$ is ordinary at \emph{every}
place of $\Q^\alg$ above $p$.  Then the abelian variety (up to isogeny) $A_f$
attached to $f$ by Eichler-Shimura theory is ordinary at $p$, and a
theorem of Mazur \cite[Proposition 4.39]{mazur} tells us that
the universal norm subgroup of $A_f(H_{s,v})$ from a ramified $\Z_p$-extension
has finite index.
\end{Rem}


\section{Completion of the proofs}
\label{fin}


Assume $D$ is odd and $\not=-3$, and that $\epsilon(p)=1$.
Fix $s>0$ and $\sigma\in\Gal(H_s/K)$.  Let $\fa$ be a proper
integral $\co_s$-ideal of norm prime to $p$ whose class in $\Pic(\co_s)$
represents $\sigma$.
Recall from \S \ref{the proof} the $p$-adic modular form $F_\sigma$ 
defined by
$$
F_\sigma= U^2 F_\sigma^{s,s}-U F_\sigma^{s,s-1}-UF_\sigma^{s-1,s}
+F_\sigma^{s-1,s-1}\ \ \in M_2(\Gamma_0(Np),\cA)\otimes_{\cA}\cB.
$$

\begin{Prop}\label{F expansion}
For every $m=m_0p^r$ with $(m_0,Np)=1$, 
\begin{eqnarray}
a_m(F_\sigma)&=&\langle c_s, T_{mp^2}(d_s^\sigma)\rangle
-\langle c_s, T_{mp}(d_{s-1}^\sigma)\rangle  
+\langle c_{s-1}, T_m(d_{s-1}^\sigma)\rangle
-\langle c_{s-1}, T_{mp}(d_s^\sigma)\rangle \nonumber\\
&=& 
\langle c_s, T_{m_0}(\bd^\sigma_{s,r+2})\rangle-
\langle c_{s-1}, T_{m_0}(\bd^\sigma_{s,r+1})\rangle\label{no tangents}.
\end{eqnarray} 
where 
$\langle\ ,\ \rangle=\langle\ ,\ \rangle_{X_0(N),H_s}
$
is the global pairing of  (\ref{global curve}) viewed as a pairing
on $J_0(N)(H_s)$,
 and $c_s$, $d_s$, $\bc_{s,r}$, and $\bd_{s,r}$ are as in \S \ref{the proof}.
Furthermore, extending the height pairing $\cB$-bilinearly
to $J_0(N)(H_s)\otimes\cB$,
$$
L_f(F_\sigma)=
(\alpha^{2}-1)\alpha^{2s}\langle z_{s},z_{s}^\sigma\rangle
$$
where $L_f$ is the linear functional on 
$M_2(\Gamma_0(Np^\infty),\cA)$ of Lemma \ref{linear lemma}
and $z_s$ is the regularized Heegner point appearing in 
Theorem \ref{main result}.
\end{Prop}

\begin{proof}
Recall, for $i,j\le s$ and any $m$,
$$
a_m(F_\sigma^{i,j})= \sum_\beta \langle c_i, d_{j,\beta}^\sigma\rangle
a_m(f_\beta)
$$
where the sum is over algebra homomorphisms $\beta:\Hecke\map{}\Q^\alg$,
$f_\beta$ is the associated primitive eigenform, and $d_{j,\beta}^\sigma$
is the projection of $d_j^\sigma\in J_0(N)(H_s)$ to $J(H_s)_\beta$.
Thus if $(m,N)=1$
$$
a_m(F_\sigma^{i,j})  
= \sum_\beta \langle c_i, \beta(T_m) d_{j,\beta}^\sigma\rangle \\
= \sum_\beta \langle c_i, T_m d_{j,\beta}^\sigma\rangle \\
= \langle c_i, T_m d_{j}^\sigma\rangle.
$$
The first claim follows easily from this and the Euler system relations
of \S \ref{Hecke action}.

For the second claim,
$$
L_f(F_\sigma)=\alpha^2 L_f(F_\sigma^{s,s})- \alpha 
L_f(F_\sigma^{s,s-1})- \alpha L_f(F_\sigma^{s-1,s})
+ L_f(F_\sigma^{s-1,s-1})
$$
by the final claim of Lemma \ref{linear lemma}.  It follows from 
the same lemma that $L_f(f_\beta)=0$ unless $f_\beta=f$ (as in 
the proof of Lemma \ref{Ann II}), while 
$L_f(f)=1-\alpha^{-2}$.  Therefore 
$$
L_f(F_\sigma^{i,j}) = (1-\alpha^{-2})
\langle c_i, d_{j,f}^\sigma\rangle = (1-\alpha^{-2})
\langle d_{i,f}, d_{j,f}^\sigma\rangle
$$
where the subscript $f$ indicates projection to the component 
$J(H_s)_{\beta_f}$ of
the algebra homomorphism $\beta_f:\Hecke\map{}\Q^\alg$ associated to $f$,
and the second equality uses the fact that $c_i-d_i=(\infty)-(0)$
is torsion in $J_0(N)(H_s)$ and that summands $J(H_s)_\beta$
are orthogonal for distinct $\beta$ (an easy consequence of
Proposition \ref{curve height}(c)).
This gives
\begin{eqnarray*}
L_f(F_\sigma) &=& (1-\alpha^{-2})\big[
\alpha^2\langle d_{s,f}, d_{s,f}^\sigma\rangle
-\alpha \langle d_{s,f}, d_{s-1,f}^\sigma\rangle
-\alpha \langle d_{s-1,f}, d_{s,f}^\sigma\rangle
+ \langle d_{s-1,f}, d_{s-1,f}^\sigma\rangle
\big] \\
&=& (1-\alpha^{-2})
\langle \alpha d_{s,f}- d_{s-1,f}, \alpha d_{s,f}^\sigma - 
d_{s-1,f}^\sigma\rangle \\
&=& (\alpha^2- 1)
\langle  \alpha^s z_s, \alpha^s z_s^\sigma\rangle
\end{eqnarray*}
as $z_s$ was defined to be $\alpha^{-s}(d_{s,f}-\alpha^{-1} d_{s-1,f})$
(in the introduction we abusively confused $h_i$ with $d_i=(h_i)-(\infty)$).
\end{proof}

As explained in \S \ref{the proof}, in each of the pairings 
of (\ref{no tangents}) the divisors
have disjoint supports, and so we may decompose 
$a_m(F_\sigma)=\sum_v a_m(F_\sigma)_v$ as a sum of local N\'eron symbols
on $X_v=X_0(N)\times_\Q H_{s,v}$ by defining
$$
a_m(F_\sigma)_v=\langle c_s, T_{m_0}(\bd^\sigma_{s,r+2})\rangle_v-
\langle c_{s-1}, T_{m_0}(\bd^\sigma_{s,r+1})\rangle_v
$$
where for each prime $v$ of $H_s$, 
$\langle\ ,\ \rangle_v=\langle\ ,\ \rangle_{X_v,\rho_{H_{s,v}}}$ 
is the local N\'eron symbol of Proposition \ref{curve height}.
We also define, for a rational prime $\ell$, 
$a_m(F_\sigma)_\ell=\sum_{v|\ell} a_m(F_\sigma)_v$.

\begin{Prop}\label{finite heights}
Suppose $(m,N)=1$.  Then
$$\sum_{\ell\not=p}a_m(F_\sigma)_\ell= 
a_{mp^{2s}}(G_{\sigma\kappa})-a_{mp^{2s+2}}(G_{\sigma\kappa} ),$$
where $G_\sigma$ is the $p$-adic modular form of Proposition \ref{L prop}.
\end{Prop}

\begin{proof} For any $\ell\not=p$,
Proposition \ref{split primes} shows that $a_m(F_\sigma)_\ell=0$
when $\epsilon(\ell)=1$, while Propositions
\ref{quaternion sum} and  \ref{nonsplit evaluation}
give an explicit formula for $a_m(F_\sigma)_\ell$ when 
$\epsilon(\ell)\not=1$. Corollary \ref{pr calc 2} 
gives an explicit formula for the right hand side.
\end{proof}

\begin{proof}[Proof of Theorem \ref{main result}]
If we define a $p$-adic modular form 
$E_\sigma\in M_2(\Gamma_0(Np^\infty),\cA)\otimes_\cA\cB$ by
$$
E_\sigma= F_\sigma- U^{2s}(1-U^2)G_{\sigma\kappa},
$$
then for every $m=m_0 p^r$ with  $(m_0,Np)=1$
Proposition \ref{finite heights} implies
$$
a_m(E_\sigma)=
\langle c_s, T_{m_0}(\bd^\sigma_{s,r+2})\rangle_p-
\langle c_{s-1}, T_{m_0}(\bd^\sigma_{s,r+1})\rangle_p.
$$
Proposition \ref{annihilation} now implies $L_f(E_\sigma)=0$, and so
$$
L_f(F_\sigma)=L_f(U^{2s}(1-U^2)G_{\sigma\kappa}).
$$
Applying Lemma \ref{linear lemma}(d) and Proposition 
\ref{F expansion}
$$
(\alpha^{2}-1)\alpha^{2s}\langle z_{s},z_{s}^\sigma\rangle_{X_0(N),H_s}
=
\alpha^{2s}(1-\alpha^2) L_f(G_{\sigma\kappa}).
$$
Summing over $\sigma$ and applying Proposition \ref{L prop},
$$
\sum_\sigma\eta(\sigma)\langle z_{s},z_{s}^\sigma\rangle_{X_0(N),H_s}
=
-\sum_\sigma\eta(\sigma)L_f(G_{\sigma\kappa})
=
-\log_p(\gamma_0)\eta(\kappa)\cdot \cL_{f,1}(\eta)
$$
for any character $\eta$ of $\Gal(H_s/K)$.  
We now view $z_s$ as an element of $J_0(N)(H_s)\otimes\cB$,
let $z_s^\vee$ be the image of $z_s$ in $J_0(N)(H_s)^\vee\otimes\cB$ under
the canonical polarization, and switch to the height
pairing $\langle\ ,\ \rangle_{J_0(N),H_s}$ of (\ref{global abelian}).
Recalling Remark \ref{Theta},
$$
\sum_\sigma\eta(\sigma)\langle z_s^\vee,z_{s}^\sigma\rangle_{J_0(N),H_s}=
\log_p(\gamma_0)\eta(\kappa)\cdot \cL_{f,1}(\eta).
$$
This completes the
proof of Theorem \ref{main result} when $s>0$.
If $\eta$ is a character of
$\Gal(H_0/K)$, then we may view $\eta$ as a character of $\Gal(H_s/K)$
for some $s>0$, and this does not change the value of $\cL_{f,1}(\eta)$.
As the $z_s$ and $z_s^\vee$ are norm compatible
\begin{eqnarray*}
\sum_{\sigma\in\Gal(H_s/K)}
\eta(\sigma)\langle z_{s}^\vee,z_{s}^\sigma\rangle_{J_0(N), H_s}
&=&\sum_{\sigma\in\Gal(H_0/K)}
\eta(\sigma)\langle z_{s}^\vee,z_{0}^\sigma\rangle_{J_0(N), H_s}\\
&=&\sum_{\sigma\in\Gal(H_0/K)}
\eta(\sigma)\langle z_0^\vee,z_{0}^\sigma\rangle_{J_0(N), H_0},
\end{eqnarray*}
so Theorem \ref{main result} holds also when $s=0$.
\end{proof}

\begin{proof}[Proof of Theorem \ref{main result II}]
If we show that
\begin{equation}\label{change of variety}
\langle y_s^\vee, y_s^\sigma\rangle_{E,H_s}
=
\langle z_s, z_s^\sigma\rangle_{J_0(N),H_s}
\end{equation}
for any $s$ then we are done, as Theorem \ref{main result} shows that
the two sides of the equality of Theorem \ref{main result II}
agree on all finite order characters.  Implicit in this
statement is that (\ref{change of variety}) holds 
for any choice of height pairing 
$\langle\ ,\ \rangle_{J_0(N),H_s}$ as in (\ref{global abelian})
(recall that the definition of (\ref{global abelian}) depends
on the possibly non-canonical choice of the local symbol 
$\langle\ ,\ \rangle_{J_0(N)_v,\rho_{H_{s,v}}}$ of Proposition 
\ref{pr height} 
for each place $v$ above $p$, and that there is a unique choice of
local symbol $\langle\ ,\ \rangle_{E_v,\rho_{H_{s,v}}}$ at every place $v$).
Fix a prime $v$ of $H_s$ and define a $\Q_p$-valued symbol 
$\langle c, d\rangle$ on pairs of degree zero divisors
on $E_v=E\times_{\Q}H_{s,v}$ with disjoint support (and $d$ rational
over $H_{s,v}$ point-by-point) by
$$
\langle c,d\rangle= \frac{1}{n}
\langle \phi^*c,\delta\rangle_{J_0(N)_v,\rho_{H_{s,v}}}
$$
where $\delta$ is a zero cycle  on $J_0(N)_v$ such that
$n\cdot d=\phi_*\delta$ for some $n$
(using the fact that $\phi_*: J_0(N)(H_{s,v})\map{}E(H_{s,v})$
has finite cokernel). 
It can be shown that the symbol 
$\langle\ ,\ \rangle$ satisfies the properties of Proposition
\ref{pr height}, and so must be the \emph{unique} symbol
$\langle\ ,\ \rangle_{E_v,\rho_{H_{s,v}}}$.  From this
one easily deduces the compatibility of the global symbols
(\ref{global abelian})
$$
\langle c,\phi_* d\rangle_{E,H_s}=\langle \phi^*c, d\rangle_{J_0(N),H_s}
$$
for $c\in E(H_s)$ and $d\in J_0(N)(H_s)$.
The equality (\ref{change of variety}) is then obvious from the 
definition of $y_s$ and $y_s^\vee$.
\end{proof}

\bibliographystyle{amsalpha}

\begin{thebibliography}{gggggg}

\bibitem[AH03]{agboola}
A. ~Agboola and B. ~Howard.
\newblock Anticyclotomic Iwasawa theory of CM elliptic curves.
\newblock Preprint. 2003.


\bibitem[AtLe70]{AL}
A.O.L. Atkin and J. Lehner.
\newblock Hecke operators on $\Gamma_0(N)$.
\newblock \emph{Math. Ann.} 185:134--160. 1970.


\bibitem[Ber95]{bert}
M. ~Bertolini.
\newblock Selmer groups and Heegner points in anticyclotomic 
$\Z_p$-extensions.
\newblock \emph{Comp. Math.} 99:153--182. 1995.

\bibitem[BD96]{bd mumford}
M. ~Bertolini and H. ~Darmon.
\newblock Heegner points on Mumford-Tate curves.
\newblock  \emph{Invent. Math.} 126:413-456. 1996.

\bibitem[Bl80]{Bloch}
S. ~Bloch.
\newblock A note on height pairings, Tamagawa numbers, and the
Birch and Swinnerton-Dyer Conjecture.
\newblock \emph{Invent. Math.} 58:65-76. 1980.




\bibitem[Con03]{conrad}
B. ~Conrad.
\newblock Gross-Zagier revisited.
\newblock In {\em Heegner Points and Rankin $L$-series},
eds. Darmon and Zhang. Mathematical Sciences Research Institute Publications.
pp.67--163. 2004.

\bibitem[Cor02]{cornut}
C. ~Cornut.
\newblock Mazur's conjecture on higher Heegner points.
\newblock {\em Invent. Math.} 148:495--523. 2002.

\bibitem[Cox89]{cox}
D. ~Cox.
\newblock {\em Primes of the Form $x^2+ny^2$}.
\newblock John Wiley and Sons. 1989.



\bibitem[Gor02]{Goren}
E. ~Goren.
\newblock {\em Lectures on Hilbert Modular Varieties and Modular
Forms}.
American Mathematical Society. 2002.


\bibitem[Gro85]{gross}
B. ~Gross.
\newblock Local heights on curves.
\newblock In {\em Arithmetic Geometry}, eds. Cornell and Silverman.
pp. 327--339. 1985.


\bibitem[GZ86]{gz}
B.~Gross and D.~Zagier.
\newblock Heegner points and derivatives of {$L$}-series.
\newblock {\em Invent. Math.}, 84:225--320. 1986.

\bibitem[Hi85]{Hida}
H. ~Hida.
\newblock A $p$-adic measure attached to the zeta function associated
with two elliptic modular forms I.
\newblock {\em Invent. Math.}, 79:159--195. 1985.

\bibitem[Hi93]{Hida book}
H. ~Hida.
\newblock \emph{Elementary theory of $L$-functions and Eisenstein series.}
\newblock London Mathematical Society. 1993.

\bibitem[How03a]{me}
B. ~Howard.
\newblock The Heegner point Kolyvagin system.
\newblock To appear in \emph{Comp. Math.}


\bibitem[KM85]{Katz-Mazur}
N. ~Katz and B. ~Mazur.
\newblock {\em Arithmetic Moduli of Elliptic Curves}.
\newblock Princeton University Press. 1985.

\bibitem[La88]{Lang}
S. ~Lang.
\newblock {\em Introduction to Arakelov Theory}.
\newblock Springer-Verlag. 1988.

\bibitem[Mann03]{mann}
W.R. ~Mann.
\newblock Elimination of quaternionic sums, appendix to \cite{conrad}.

\bibitem[Maz72]{mazur}
B. ~Mazur.
\newblock Rational points of abelian varieties with values in towers of
number fields.  
\newblock {\em Invent. Math.}, 18:183--266. 1972.

\bibitem[MR02]{MR}
B. ~Mazur and K. ~Rubin.
\newblock Elliptic curves and class field theory.
\newblock \emph{Proceedings of the ICM, Beijing 2002}, vol. 2, pp.185--196.

\bibitem[Mil86]{milne}
J.S. ~Milne.
\newblock Abelian Varieties.
\newblock In {\em Arithmetic Geometry}, eds. Cornell and Silverman.
pp. 103--150. 1985.


\bibitem[Nek95]{nekovar}
J. ~Nekov\'a\v{r}.
\newblock On the $p$-adic height of Heegner cycles.
\newblock {\em Math. Ann.}, 302:609--686. 1995.


\bibitem[PR87a]{pr1}
B. ~Perrin-Riou.
\newblock Points de Heegner et d\'eriv\'ees de fonctions $L$
$p$-adiques.
\newblock {\em Invent. Math.}, 89:455--510. 1987.

\bibitem[PR87b]{pr2}
B.~Perrin-Riou.
\newblock Fonctions {$L$} $p$-adiques, th\'{e}orie d'{I}wasawa et points de
  {H}eegner.
\newblock {\em Bull. Soc. Math. France}, 115:399--456. 1987.

\bibitem[PR88]{pr0}
B. ~Perrin-Riou.
\newblock Fonctions $L$ $p$-adiques associ\'ees \`a une forme modulaire
et \`a un corps quadratique imaginaire.
\newblock {\em J. London Math. Soc.} (2) 38:1--32. 1988.

\bibitem[PR91]{pr3}
B. ~Perrin-Riou.
\newblock Th\'eorie d'Iwasawa et hauteurs $p$-adiques (cas des vari\'et\'es
ab\'eliennes).
\newblock Unpublished. 1991.

\bibitem[PR92]{pr4}
B. ~Perrin-Riou.
\newblock Th\'eorie d'Iwasawa et hauteurs $p$-adiques.
\newblock {\em Invent. Math.} 109:137--185. 1992.


\bibitem[Sch87]{Schneider}
P. ~Schneider.
\newblock Arithmetic of formal groups and applications I:
Universal norm subgroups.
\newblock {\em Invent. Math.} 87:587--602. 1987.

\bibitem[SeTa69]{st}
J.-P. ~Serre and J. ~Tate.
\newblock Good reduction of abelian varieties.
\newblock {\em Ann. Math.} 88:492--517. 1968.

\bibitem[Sha85]{shatz}
S.S. ~Shatz.
\newblock Group schemes, formal groups, and $p$-divisible groups.
\newblock In {\em Arithmetic Geometry}, eds. Cornell and Silverman,
pp. 29--78. 1985.


\end{thebibliography}

\end{document}